\documentclass[a4paper,final]{siamltex}

\usepackage{graphicx}
\usepackage[english]{babel}
\usepackage{amsmath}
\usepackage{amssymb}
\usepackage{url}
\usepackage{multirow}
\usepackage{hyperref}

\usepackage{graphicx}
\usepackage{tikz}
\usepackage{float}

\newcommand{\ii}{{\mathbf{i}}}
\newcommand{\jj}{\mathbf{j}}
\newcommand{\nn}{{\mathbf{n}}}

\newcommand{\xx}{{\mathbf{x}}}

\newcommand{\yy}{{\mathbf{y}}}
\renewcommand{\aa}{{\mathbf{a}}}

\newcommand{\hh}{{\mathbf{h}}}

\newcommand{\etab}{\boldsymbol{\eta}}
\newcommand{\bm}[1]{\mathbf{#1}}

\newcommand{\cA}{\mathcal{A}}
\newcommand{\cC}{\mathcal{C}}
\newcommand{\cE}{\mathcal{E}}
\newcommand{\cF}{\mathcal{F}}
\newcommand{\cG}{\mathcal{G}}
\newcommand{\cK}{\mathcal{K}}
\newcommand{\cI}{\mathcal{I}}

\newcommand{\cM}{\mathcal{M}}
\newcommand{\cN}{\mathcal{N}}

\newcommand{\cS}{\mathcal{S}}
\newcommand{\cU}{\mathcal{U}}
\newcommand{\cX}{\mathcal{X}}

\newcommand{\e}{{\mathrm{e}}}

\newcommand{\kk}{\mathsf{k}}

\newcommand{\R}{\mathbb{R}}

\DeclareMathOperator{\sspan}{span}

\newtheorem{example}{Example}

    \graphicspath{{figures/}}
\date{\today}

\graphicspath{{./}{figures/}}

\usepackage{mathptmx}      

\begin{document}

\title{Anti-Gauss cubature rules with applications to Fredholm integral equations on the square}

\author{Patricia D\'iaz de Alba\thanks{Department of Mathematics and Computer Science, University of Cagliari, via Ospedale 72, 09124 Cagliari, Italy, \texttt{patricia.diazda@unica.it, fermo@unica.it, rodriguez@unica.it}}
\and
Luisa Fermo\footnotemark[1] 
\and
Giuseppe Rodriguez\footnotemark[1]
}

\maketitle

\begin{abstract}
The purpose of this paper is to develop the anti-Gauss cubature rule for approximating integrals defined on the square whose integrand function may have algebraic singularities at the boundaries. An application of such a rule to the numerical solution of Fredholm integral equations of the second-kind is also explored.
The stability, convergence, and conditioning of the proposed  Nystr\"om-type method are studied. The numerical solution of the resulting dense linear system is also investigated and several numerical tests are presented. 
\end{abstract}



\section{Introduction}

Let us consider the integral
\begin{equation*}
\cI(f)=\int_\cS f_1(\xx) \, d\xx,
\end{equation*} 
where $\cS:=[-1,1] \times [-1,1]$, $\xx=(x_1,x_2)$, and 
$f_1$ is an integrable bivariate function which may have algebraic
singularities on the boundary of $\cS$. We deal with such
singularities by writing
\begin{equation}\label{int}
\cI(f)=\int_\cS f(\xx) w(\xx) \, d\xx
= \int_{-1}^1 \int_{-1}^1 f(x_1,x_2) w_1(x_1) w_2(x_2) \, dx_1 dx_2,
\end{equation} 
that is, by factoring $f_1$ as the product of a function $f$ which is
sufficiently smooth on $\cS$ and a weight function 
\begin{equation}\label{w}
w(\xx) =w_1(x_1) w_2(x_2),
\end{equation}
with $w_i(x_i)=(1-x_i)^{\alpha_i} (1+x_i)^{\beta_i}$ for $\alpha_i,\beta_i>-1$
and $i=1,2$. Basically, we deal with integrand functions having endpoint
singularities which can be explicitly extracted and confined into a weight
function. This approach allows for constructing specific orthogonal
polynomials and, then, Gauss quadrature rules that can be computed efficiently,
avoiding techniques which use smoothing transformations \cite{Monegato1999} or
meshes adapted to singularities \cite{Kaneko1994}.

For the numerical approximation of the integral \eqref{int}, we may opt for two
alternative techniques; see \cite{Cools,Stroud}. 
The first one, known as the ``indirect'' approach,  is based on the approximation of
each one-dimensional integral in \eqref{int} by a well-known quadrature rule.
This procedure takes advantage of the deep study and exploration on univariate rules, compared with the multivariate ones.
In \cite{OccorsioRusso2011}, the authors  propose to approximate integrals of type \eqref{int} by a cubature formula obtained as a tensor product of two
Gaussian rules; see also formula \eqref{Gn} in Section~\ref{sec:cubature}.
They study the formula in suitable weighted spaces, prove its convergence and
stability, and  provide a lower bound for the order of convergence. Such bounds depend on the
smoothness properties of the integrand function $f$ and involve a constant
independent of $f$ and the number of nodes.

The second approach, which can be considered ``direct'', consists of
constructing true bivariate cubature schemes from scratch.
This case is more involved. Indeed, it is well known that Gaussian cubature
rules based on bivariate orthogonal polynomials exist only in few cases; see
for instance \cite{SchmidXu,Xu2012}. 
In \cite{Mysovskikh}, an example is given where the collocation nodes are
obtained as zeros of particular bivariate orthogonal polynomials; see
also~\cite{dunkl_xu_2014,Morrow,Schmid1978,Xu2015}.

In the initial part of the present paper, we point our attention to the
``indirect'' approach and develop an anti-Gaussian cubature rule as a
tensor product of two anti-Gaussian univariate formulae. 
Anti-Gauss rules were introduced for the first time in \cite{Laurie1}, where
Dirk Laurie estimated the error incurring in Gaussian integration
by halving the difference between the values of an $n$-point Gauss rule and
a new $(n+1)$-point formula. The newly developed quadrature rule, when applied
to polynomials of particular degree, gives an error equal in magnitude to that
of the $n$-point Gauss rule, but opposite in sign. For this property of the
error, the formula was named anti-Gaussian rule.
After Laurie, many other authors investigated such rules and proposed new
generalizations; see, for example,
\cite{Notaris2018,Notaris2022,ReichelSpalevic2021,Reichel2022}. However,
according to our knowledge, they have been investigated in the bivariate case
only on the real semi-axis \cite{Djukic2}. Here, we present for the first time
anti-Gaussian cubature formulae on bounded domains (the square for
simplicity) whose  utility is twofold.
On the one hand, they allow one to build new cubature rules, namely, averaged
or stratified cubature formulae, which are characterized by a higher accuracy
and smaller computational cost.  On the other hand, they provide numerical
estimates for the error of the Gaussian cubature rule for a fixed number of
points.
This leads to determining the number of points required to reach a
prescribed accuracy in the integral approximation.
The estimates so obtained are independent of unknown constants and are not
asymptotic.

In the second part of the paper, we apply anti-Gauss rules to the
numerical solution of the integral equation 
\begin{equation}\label{Fredholm}
(I-K)f=g,
\end{equation}
where $f$ is the bivariate function to be recovered, defined on the square
$\cS$, $I$ is the identity operator, and $g$ is a given right-hand
side.
The integral operator $K$ is defined by
\begin{equation*}
(Kf)(\yy)=\int_{\cS}  k(\xx,\yy) f(\xx) w(\xx) d\xx,
\end{equation*}
where $\xx=(x_1,x_2)$ and $\yy=(y_1,y_2)$ belong to $\cS$,
the kernel function $k$ defined on $\cS \times \cS$ is known,
$d\xx=dx_1dx_2$,
and $w$ is the weight function given in \eqref{w}. Defining the function $w$ as
the product of two classical Jacobi weights aims at accounting for possible
algebraic singularities 
with respect to the integration variable $\bm{x}$, 
at the boundary of the domain, 
of the solution and the kernel.
Let us assume that, in addition to this, the right-hand side $g$ and the
kernel $k$, with respect to the external variable $\bm{y}$,
have a low smoothness at the boundary of the square, i.e.,
their derivatives are singular at some boundary points. Since the solution
inherits such a regularity, to take care of this behaviour we consider the
equation in suitable weighted spaces, by introducing an additional weight
function $u(\bm{y})$.

Equation \eqref{Fredholm} arises in several problems related to electromagnetic
scattering, aerodynamics, computer graphics and mathematical physics.
Examples are the radiosity equation~\cite{Atkinson2000} and the rendering
equation~\cite{Kajiya1986}. 
In view of such applications, different numerical approaches have been
developed for the solution of
equation \eqref{Fredholm}, such as weighted Nystr\"om type methods
\cite{Laguardia2022,OccorsioRusso2011,OccorsioRusso2016}, integral mean value
methods \cite{Ma2015}, Galerkin methods \cite{Han2002,Jebreen2020}, collocation
methods \cite{Alipanah2011,Hat-Var2011,Mirzaee2015}, and wavelets methods
\cite{Wang2005}.

Recently, much attention has been devoted, in the one dimensional case, to
numerical techniques that exploit the advantages of anti-Gaussian type
formulae; see, for instance,~\cite{DFR2020} or~\cite{FRRS2022}.
In light of the numerical accuracy that such formulae are able to
reach, in this paper we introduce a weighted Nystr\"om method based on
anti-Gauss cubature formulae to solve equation \eqref{Fredholm} in suitable
weighted spaces, and investigate its stability and convergence.
We underline that blending Gauss-type quadrature formulae with
weighted spaces allow us to treat possible singularities or low smoothness at
the boundary for both the solution and the right-hand side, obtaining a
theoretical error of the order of the best polynomial approximation. 
We further propose to combine the above method with the Nystr\"om
method based on the Gauss rule, presented in \cite{OccorsioRusso2011}.   
This combination allows us to construct two Nystr\"om interpolants that, under
suitable assumptions, bracket the solution of the integral equation.  As a
consequence, an average of the two numerical solution produces a better
accuracy.

The numerical solution of the resulting linear system is also investigated.
The system
is characterized by a dense coefficients matrix and by a dimension which
becomes large when the functions involved have a low degree of smoothness. The
iterative solution by the GMRES method is investigated and the special
case of a separable kernel is also considered.  

Summarizing, three are the main novelty of the paper. First, we
construct an anti-Gauss cubature rule for approximating integrals defined on
the square whose integrand function may have algebraic
singularities at the boundary. Second, we develop a global approximation method
of the Nystr\"om type based on such formula and, by combining it with a
Gauss-based Nystr\"om method, provide an averaged Nystr\"om interpolant and an
error estimate.
Third, we explore various approaches for the numerical solution of the
resulting dense linear system and compare them.
According to our knowledge, this is the first paper in which such a
solution is analyzed.

The paper is organized as follows. In Section \ref{sec:cubature}, we introduce
the anti-Gauss cubature rule and investigate its properties with
Proposition~\ref{prop:anti_prop}. Under suitable assumptions, we extend the
bracketing property to a general function $f$ (Theorem~\ref{theorem2}) and
provide simpler assumptions in the Chebychev case; see
Corollaries~\ref{corollary1} and~\ref{corollary2}.
Section \ref{sec:nystrom} describes a Nystr\"om method based
on the Gauss and anti-Gauss rules, and show that the two corresponding Nystr\"om
interpolants bracket the solution of the integral equation, suggesting that a
better accuracy can be obtained by taking the average of the two interpolants.
In Section \ref{sec:system}, we analyze the linear systems that yield the
interpolants and solve them by optimized versions of the GMRES
iterative method.
In particular, we investigate the special case of a separable kernel.
Finally, Section \ref{sec:tests} presents the results of a numerical
experimentation on the new cubature rule and Nystr\"om method, supporting the
theoretical analysis, while Section \ref{sect:Concl} contains some conclusions and perspectives for future work.

\section{Cubature rules} \label{sec:cubature}

Let us consider the integral \eqref{int},
with the weight function $w$ defined in \eqref{w}.
To obtain a numerical approximation, we apply to each nested weighted integral
the optimal Gauss-Jacobi rule
\begin{equation}\label{gauss1d}
G^{(\ell)}_{n}(g) = \sum_{j=1}^{n} \lambda_{j}^{(\ell)}
g(x_{j}^{(\ell)}),
\end{equation}
where $g(x)$ is a univariate function defined on $[-1,1]$, $\lambda_j^{(\ell)}$
is the $j$th Christoffel number with respect to the weight $w_\ell(x)$
appearing in the integral, and $x_{j}^{(\ell)}$ is the $j$th zero of the
monic polynomial $p_{n}^{(\ell)}(x)$ orthogonal with respect to the same
weight, for $\ell=1,2$.

To ease exposition, we recall that $p_{n}^{(\ell)}(x)$ satisfies the
well-known three-term recurrence relation
$$
\begin{cases}
p^{(\ell)}_{-1}(x)=0, \quad p^{(\ell)}_0(x)=1, \\
p^{(\ell)}_{j+1}(x)=(x-a^{(\ell)}_j) p^{(\ell)}_j(x)-b^{(\ell)}_j p^{(\ell)}_{j-1}(x), \quad j=0,1,2,\ldots,
\end{cases}
$$
where the coefficients $a^{(\ell)}_j$ and $b^{(\ell)}_j$ are given by
\begin{equation}\label{ortcoef}
\begin{aligned}
a^{(\ell)}_j&= \frac{\beta_\ell^2-\alpha_\ell^2}{(2j+\alpha_\ell+\beta_\ell)(2j+\alpha_\ell+\beta_\ell+2)},
&j \geq 0, \\
b^{(\ell)}_0 &= \frac{2^{\alpha_\ell+\beta_\ell+1} \Gamma(\alpha_\ell+1) \Gamma(\beta_\ell+1)}{\Gamma(\alpha_\ell+\beta_\ell+2)}, \\
b^{(\ell)}_j&= \frac{4j(j+\alpha_\ell)(j+\beta_\ell)(j+\alpha_\ell+\beta_\ell)}{(2j+\alpha_\ell+\beta_\ell)^2
((2j+\alpha_\ell+\beta_\ell)^2-1)}, &j \geq 1.
\end{aligned}
\end{equation}

It is well known \cite{Gautschi04} that the zeros of $p_{n}^{(\ell)}(x)$ can be
efficiently computed as the eigenvalues of the Jacobi matrix associated to the
polynomials, while the Christoffel numbers are the squared first components of
the normalized eigenvectors of the same matrix.

Let us go back to the approximation of \eqref{int}.
By using $n_1$ points in the integral with the
differential $dx_1$ and $n_2$ nodes in that with $dx_2$, we obtain the
$(n_1 \times n_2)$-point Gauss cubature rule
\begin{equation}\label{Gn}
\cG_{n_1,n_2}(f)=\sum_{j_1=1}^{n_1}
\sum_{j_2=1}^{n_2}\lambda_{j_1}^{(1)} \lambda_{j_2}^{(2)}
f(x_{j_1}^{(1)},x_{j_2}^{(2)}).
\end{equation}
Denoting by $R^{(G)}_{n_1,n_2}(f)$ the remainder term for the integral, i.e., 
\begin{equation}\label{erreg}
\cI(f)=\cG_{n_1,n_2}(f)+R^{(G)}_{n_1,n_2}(f),
\end{equation}
it is immediately to observe that the interpolatory scheme \eqref{Gn} is such
that $$R^{(G)}_{n_1,n_2}(p)=0, \qquad  \forall p \in \mathbb{P}_{2n_1-1,2n_2-1},$$
where $\mathbb{P}_{k,\ell}$ is the set of all bivariate polynomials of the type
$$
p(x,y)=\sum_{i=0}^k \sum_{j=0}^\ell a_{ij} x^i y^j, 
\qquad a_{ij} \in \mathbb{R},
$$
whose degree is at most $k$ in the variable $x$ and at most $\ell$ in $y$.

In \cite[Proposition 2.2]{OccorsioRusso2011}, estimates for the error
$R^{(G)}_{n_1,n_2}(f)$ are given in terms of the smoothness properties of the
function $f$. Basically, the cubature error goes to zero as the error of best
polynomial approximation for $f$. Here, we want to provide an estimate for such
error by using stratified schemes. This approach is well consolidated in the
one-dimensional case through the well known Gauss-Kronrod formulae
\cite{Notaris2}, the anti-Gauss quadrature rules \cite{Laurie1}, and their
recent extensions \cite{Djukic1,Djukic2,ReichelSpalevic2021,Spalevic2007}.

To this end, we introduce the anti-Gaussian cubature scheme
\begin{equation}\label{A}
\cA_{n_1+1,n_2+1}(f)=\sum_{j_1=1}^{n_1+1} \sum_{j_2=1}^{n_2+1} \mu_{j_1}^{(1)} \mu_{j_2}^{(2)} f(\eta_{j_1}^{(1)},\eta_{j_2}^{(2)}),
\end{equation}
where $\mu_i^{(\ell)}$ is the $i$th anti-Gaussian quadrature weight for
$\ell=1,2$, and  $\eta_i^{(\ell)}$ is the $i$th zero of the polynomial $q^{(\ell)}_{n_\ell+1}(x)=p^{(\ell)}_{n_\ell+1}(x)-b^{(\ell)}_{n_\ell} p^{(\ell)}_{n_\ell-1}(x)$, with $\ell=1,2.$
Anti-Gaussian cubature formulae and related generalizations have been very
recently investigated in \cite{DFermoM2023} for the Laguerre weight.

Similarly to \eqref{gauss1d} and \eqref{Gn}, such a cubature rule is
constructed by a tensor product of two univariate anti-Gauss
rules~\cite{Laurie1}, which we denote by $A^{(\ell)}_{n_\ell+1}$, $\ell=1,2$.
Therefore, the zeros $\{\eta_i^{(\ell)}\}_{i=1}^{n_\ell+1}$
are the eigenvalues of the matrix
$$
{\Psi}^{(\ell)}_{n_\ell+1} =
\begin{bmatrix}
J^{(\ell)}_{n_\ell} & \sqrt{2 b^{(\ell)}_{n_\ell} }\mathbf{e}_{n_\ell} \\
\sqrt{2 b^{(\ell)}_{n_\ell}} \mathbf{e}^T_{n_\ell} & a^{(\ell)}_{n_\ell} \\
\end{bmatrix}, \text{ with } J^{(\ell)}_{n_\ell}{=}
\begin{bmatrix}
a^{(\ell)}_0 & \sqrt{b^{(\ell)}_1} \\
\sqrt{b^{(\ell)}_1} & a^{(\ell)}_1 & \ddots \\
&  \ddots & \ddots & \sqrt{b^{(\ell)}_{n_\ell-1}} \\
&  & \sqrt{b^{(\ell)}_{n_\ell-1}} &  a^{(\ell)}_{n_\ell-1} \\
\end{bmatrix}
$$ 
and  $\mathbf{e}_{n_\ell}=(0,0,\dots,1)^T \in \mathbb{R}^{n_\ell}$. The coefficients $\{\mu_i^{(\ell)}\}_{i=1}^{n_\ell+1}$ are determined as $\mu^{(\ell)}_i=b^{(\ell)}_0 \, (v^{(\ell)}_{i,1})^2, $
where $b^{(\ell)}_0$ is defined in \eqref{ortcoef} and $v^{(\ell)}_{i,1}$ is
the first entry of the normalized eigenvector associated to the eigenvalue $\eta^{(\ell)}_i$.

We remark that for the computation of the eigenvalues and eigenvectors we can resort to the algorithm proposed by Golub and Welsch in \cite{GW}.
It is based on the QR factorization with a Wilkinson-like shift and has a
computational cost $c n_\ell^2+O(n_\ell)$, $\ell=1,2$, where $c$ is a small positive constant which does not depend on $n_\ell$.

Let us mention that, by definition, all the weights are positive and the zeros
interlace the nodes of the Gauss rule \cite{Laurie1}, i.e.,
$\eta_1^{(\ell)}<x_1^{(\ell)}<\eta_2^{(\ell)}<x_2^{(\ell)}<\cdots<x^{(\ell)}_{n_\ell}<\eta^{(\ell)}_{n_\ell+1}.$
Moreover, the anti-Gauss nodes $\eta^{(\ell)}_i$ belong to
the interval $[-1,1]$ when
\begin{equation}\label{condparam}
\begin{cases}
\alpha_\ell \geq -\frac{1}{2}, \\
\beta_\ell \geq -\frac{1}{2}, \\
(2 \alpha_\ell+1)(\alpha_\ell+\beta_\ell+2)+\frac{1}{2}(\alpha_\ell+1)(\alpha_\ell+\beta_\ell)(\alpha_\ell+\beta_\ell+1) \geq 0, \\
(2 \beta_\ell+1)(\alpha_\ell+\beta_\ell+2)+\frac{1}{2}(\beta_\ell+1)(\alpha_\ell+\beta_\ell)(\alpha_\ell+\beta_\ell+1) \geq 0.
\end{cases}
\end{equation}

We remark that conditions \eqref{condparam} are satisfied by some classical
Jacobi weights, in particular by the Legendre weight
($\alpha_\ell=\beta_\ell=0$) and the Chebychev weights of the first
($\alpha_\ell=\beta_\ell=-1/2$), second ($\alpha_\ell=\beta_\ell=1/2$), third
($\alpha_\ell=-1/2$, $\beta_\ell=1/2$), and fourth kind ($\alpha_\ell=1/2$,
$\beta_\ell=-1/2$). However, the corresponding nodes might include the endpoints of the integration interval.
This is true, for example, for the Chebychev weight of the first (${\eta}_1=-1$
and $\eta_{n_\ell+1}=1$), third ($\eta_{n_\ell+1}=1$), and fourth kind
($\eta_1=-1$).
In the case of Chebychev polynomials of the first kind an explicit form for the
nodes and weights have been given in \cite[Theorem 2]{DFR2020}.
From now on, we assume that conditions \eqref{condparam} are satisfied.

Denoting by $R^{(A)}_{n_1+1,n_2+1}(f)$ the related cubature error, i.e.,
\begin{equation}\label{errea}
\cI(f)=\cA_{n_1+1,n_2+1}(f)+R^{(A)}_{n_1+1,n_2+1}(f),
\end{equation}
we have the following proposition, which has been proved
in~\cite[Proposition~1]{DFermoM2023} for the Laguerre weight on $[0,\infty)$.
\begin{proposition}\label{prop:anti_prop}
The error of the anti-Gauss cubature scheme \eqref{A} has the following property
\begin{equation}\label{errorA}
R^{(A)}_{n_1+1,n_2+1}(p)=-R^{(G)}_{n_1,n_2}(p), \quad \forall p \in \mathbb{P}_{2n_1+1,2n_2-1} \cup \mathbb{P}_{2n_1-1,2n_2+1}.
\end{equation}
\end{proposition}

\begin{proof}
The proof follows the same line as that of~\cite[Proposition~1]{DFermoM2023}.
\end{proof}
\smallskip

It is worth noting that, as anticipated in the Introduction, the change
of sign in the cubature error stated in \eqref{errorA} motivated Dirk Laurie
to refer to \emph{Anti-Gaussian quadrature formulas} in \cite{Laurie1}.

Hence, by virtue of \eqref{errorA}, we can immediately deduce some important
features of the rule $\cA_{n_1+1,n_2+1}$:
\begin{enumerate}

\item If $p \in \mathbb{P}_{2n_1-1,2n_2-1}$, then $R^A_{n_1+1,n_2+1}(p)=0$.

\item If $p \in \mathbb{P}_{2n_1+1,2n_2-1} \cup \mathbb{P}_{2n_1-1,2n_2+1}$,
the Gauss and the anti-Gauss cubature rules provide an interval containing the
exact integral $\cI(p)$. Indeed, it either holds
\begin{equation*}
\cA_{n_1+1,n_2+1}(p) \leq \cI(p) \leq \cG_{n_1,n_2}(p)
\quad \text{or}  \quad
\cG_{n_1,n_2}(p) \leq \cI(p) \leq \cA_{n_1+1,n_2+1}(p).
\end{equation*}

\item For every polynomial $p\in\mathbb{P}_{2n_1+1,2n_2-1} \cup
\mathbb{P}_{2n_1-1,2n_2+1}$, it holds
$$
\cI(p)=
\frac{1}{2}\left[\cG_{n_1,n_2}(p)+\cA_{n_1+1,n_2+1}(p)\right].
$$
This means that the convex combination of the two cubature formulae at the
right-hand side is a cubature formula more accurate than the Gauss rule.
From now on, we will denote it by
$$
\cG^{Avg}_{2n_1+1,2n_2+1}(f)=
\frac{1}{2}\left[\cG_{n_1,n_2}(f)+\cA_{n_1+1,n_2+1}(f)\right],
$$
and we will call it \emph{averaged Gauss cubature formula}. It has positive
weights and  involves $(2n_1+1)\times(2n_2+1)$ real and distinct nodes.

\item By using the scheme $\cG^{Avg}_{2n_1+1,2n_2+1}$, we can estimate the error $R^{(G)}_{n_1,n_2}$ as
\begin{align}\label{erre1}
R^{(G)}_{n_1,n_2}=\cI(f)-\cG_{n_1,n_2}(f) & \simeq
\cG^{Avg}_{2n_1+1,2n_2+1}(f)-\cG_{n_1,n_2}(f) \nonumber \\
&= \frac{1}{2}\left[\cA_{n_1+1,n_2+1}(f)-\cG_{n_1,n_2}(f)\right]
=:R^{[1]}_{n_1,n_2}(f).
\end{align}
\end{enumerate}

The computational complexity required for the computation of nodes and weights
of $\cG^{Avg}_{2n_1+1,2n_2+1}$ is $2cn_\ell^2+2O(n_\ell)$, which halves
the cost involved for the Gauss rule $\cG_{2n_1,2n_2}$, that is,
$4cn_\ell^2+2O(n_\ell)$.

We recall that the anti-Gauss cubature rule \eqref{A} is a stable formula. This
means that if we look at the rule as a linear functional
$\cA_{n_1+1,n_2+1}: \cX \to \mathbb{R}$ where $\cX$ is
a Banach space, then
$$ \sup_{n_1,n_2} \|\cA_{n_1+1,n_2+1}\| < \infty.$$
This is a consequence of the stability of the univariate anti-Gauss quadrature 
rule, which has also been proved in weighted spaces equipped with the uniform
norm in \cite{DFR2020}, under suitable assumptions; see also \cite{FRRS2022},
where such assumptions are relaxed.

In the univariate case it has been proved, under rather restrictive
assumptions on the integrand function $f$, 
that the Gauss and the anti-Gauss quadrature rules
bracket the integral $I(f)$; see~\cite[Equations (26)-(28)]{crs1999},
\cite[p.~1664]{fmrr13b}, and \cite[Theorem 3.1]{PranicReichel}.
The same result has been proved under much less limiting assumptions
in~\cite[Corollary~1]{DFR2020}, for the solution of second-kind integral
equations.

In the following, we extend the bracketing condition to bivariate integrals,
that is, we give assumptions for which property 2) is valid for a general
function $f$ of two variables.

Let us expand the integrand function $f(\xx)$ in terms of the polynomials
$
p_{n_1,n_2}(\xx)=p^{(1)}_{n_1}(x_1)p^{(2)}_{n_2}(x_2),
$
orthogonal with respect to the weight function $w(\xx)$, in the form
\begin{equation}\label{expansion2d}
f(\xx) = \sum_{i=0}^\infty \sum_{j=0}^\infty \alpha_{i,j} p_{i,j}(\xx),
\end{equation}
where
$$
\alpha_{i,j} = \left(b_0^{(1)} b_0^{(2)}\right)^{-\frac{1}{2}}
\int_{\cS} f(\xx) p_{i,j}(\xx) w(\xx) \, d\xx.
$$

\begin{theorem}\label{theorem2}
Let us assume that the coefficients $\alpha_{i,j}$ in \eqref{expansion2d}
converge to zero sufficiently rapidly, and the following relation holds true
$$
(\cI-\cG_{n_1,n_2})(f) = -S_{n_1,n_2} + \cE^{(1)}_{n_1,n_2}, \quad
(\cI-\cA_{n_1+1,n_2+1})(f) = S_{n_1,n_2} + \cE^{(2)}_{n_1,n_2}
$$
with
\begin{equation}\label{condsign}
\max(|\cE^{(1)}_{n_1,n_2}|,|\cE^{(2)}_{n_1,n_2}|)<|S_{n_1,n_2}|,
\end{equation}
where
$$
S_{n_1,n_2} = \sqrt{b_0^{(2)}} \sum_{i=2n_1}^{2n_1+1} \alpha_{i,0}
	G^{(1)}_{n_1}(p^{(1)}_i)
+ \sqrt{b_0^{(1)}} \sum_{j=2n_2}^{2n_2+1} \alpha_{0,j}
	G^{(2)}_{n_2}(p^{(2)}_j),
$$
with $G^{(\ell)}_{n_\ell}$ defined by \eqref{gauss1d}.
The terms $\cE^{(1)}_{n_1,n_2}$ and $\cE^{(2)}_{n_1,n_2}$ depend on both $f$
and the quadrature formulae involved; their expression will be given in the
proof.

Then, either
\begin{equation*}
\cG_{n_1,n_2}(f) \leq \cI(f) \leq \cA_{n_1+1,n_2+1}(f) \quad \text{or} \quad
\cA_{n_1+1,n_2+1}(f) \leq \cI(f) \leq \cG_{n_1,n_2}(f).
\end{equation*}
\end{theorem}

\begin{proof}
From \eqref{expansion2d},
$
\cI(f) = \alpha_{0,0} \left(b_0^{(1)} b_0^{(2)}\right)^{\frac{1}{2}}.
$
Substituting \eqref{expansion2d} in \eqref{Gn} yields
$$
\cG_{n_1,n_2}(f) = \sum_{i=0}^{\infty}
\sum_{j=0}^{\infty} \alpha_{i,j} G^{(1)}_i G^{(2)}_j,
$$
where $G^{(\ell)}_i=G^{(\ell)}_{n_\ell}(p^{(\ell)}_i)$, $\ell=1,2$.
Then, exploiting the degree of exactness of $G^{(\ell)}_i$ we obtain
$
(\cI-\cG_{n_1,n_2})(f) = -S_{n_1,n_2} + \cE^{(1)}_{n_1,n_2},
$
with
$$
\begin{aligned}
\cE^{(1)}_{n_1,n_2} =
&-\sum_{i=2n_1}^{2n_1+1} \sum_{j=2n_2}^{2n_2+1} \alpha_{i,j}
	G^{(1)}_i G^{(2)}_j
-\sum_{i=2n_1+2}^{\infty} \left[ \alpha_{i,0} \sqrt{b_0^{(2)}}
	+\sum_{j=2n_2}^{2n_2+1} \alpha_{i,j} G^{(2)}_j \right] G^{(1)}_i \\
&-\sum_{i=2n_1+2}^{\infty} \sum_{j=2n_2+2}^{\infty} \alpha_{i,j}
	G^{(1)}_i G^{(2)}_j
-\sum_{j=2n_2+2}^{\infty} \left[ \alpha_{0,j} \sqrt{b_0^{(1)}}
	+\sum_{i=2n_1}^{2n_1+1} \alpha_{i,j} G^{(1)}_i \right] G^{(2)}_j.
\end{aligned}
$$

Now, substituting \eqref{expansion2d} in \eqref{A} leads to
$$
\cA_{n_1+1,n_2+1}(f) = \sum_{i=0}^{\infty}
\sum_{j=0}^{\infty} \alpha_{i,j} A^{(1)}_i A^{(2)}_j,
$$
where $A^{(\ell)}_i=A^{(\ell)}_{n_\ell+1}(p^{(\ell)}_i)$, $\ell=1,2$.
The definition of the anti-Gauss rule implies that
$
A^{(\ell)}_{n_\ell+1}(p) = 2I(p) - G^{(\ell)}_{n_\ell}(p)
= -G^{(\ell)}_{n_\ell}(p),
$
for any polynomial $p$ of degree larger than zero and smaller or equal to
$2n_\ell+1$.
By applying this property and a similar argument as before, we have
$
(\cI-\cA_{n_1+1,n_2+1})(f) = S_{n_1,n_2} + \cE^{(2)}_{n_1,n_2},
$
with
$$
\begin{aligned}
\cE^{(2)}_{n_1,n_2} =
&-\sum_{i=2n_1}^{2n_1+1} \sum_{j=2n_2}^{2n_2+1} \alpha_{i,j}
	G^{(1)}_i G^{(2)}_j
-\sum_{i=2n_1+2}^{\infty} \left[ \alpha_{i,0} \sqrt{b_0^{(2)}}
	-\sum_{j=2n_2}^{2n_2+1} \alpha_{i,j} G^{(2)}_j \right] A^{(1)}_i \\
&-\sum_{i=2n_1+2}^{\infty} \sum_{j=2n_2+2}^{\infty} \alpha_{i,j}
	A^{(1)}_i A^{(2)}_j
-\sum_{j=2n_2+2}^{\infty} \left[ \alpha_{0,j} \sqrt{b_0^{(1)}}
	-\sum_{i=2n_1}^{2n_1+1} \alpha_{i,j} G^{(1)}_i \right] A^{(2)}_j.
\end{aligned}
$$

The above relations show that when assumption \eqref{condsign} is satisfied,
there is a change of sign in the errors produced by both the Gauss rule
and anti-Gauss one. This proves the assertion.
\end{proof}

The assumption \eqref{condsign} is undoubtedly restrictive, but it is only a
sufficient condition for the bracketing of the solution.
In \cite[Corollary~1]{DFR2020} a less restrictive assumption has been given, in
the univariate case, for the Chebychev weight of the first kind.
The following corollary extends that result to bivariate integrals.

\begin{corollary}\label{corollary1}
Let $\alpha_i=\beta_i=-\frac{1}{2}$ in \eqref{w}. Then, if
\begin{equation*}
\max(|\tilde{\cE}^{(1)}_{n_1,n_2}|,|\tilde{\cE}^{(2)}_{n_1,n_2}|)< |\alpha_{2n_1,0}+\alpha_{0,2n_2} |,
\end{equation*}
holds true for $n_1$ and $n_2$ large enough, where
\begin{equation*}
\begin{aligned}
\tilde{\cE}^{(1)}_{n_1,n_2}& = \sqrt{2} \alpha_{2n_1,2n_2}+
\sum_{k_1=2}^{\infty} (-1)^{k_1} \left(\alpha_{2n_1 k_1,0}- \sqrt{2}\,
\alpha_{2n_1 k_1,2 n_2}  \right) \\
 \phantom{=} &+ \sqrt{2}
\sum_{k_1=2}^{\infty}\sum_{k_2=2}^{\infty}  (-1)^{k_1+k_2}  \alpha_{2 n_1 k_1,
2 n_2 k_2} + \sum_{k_2=2}^{\infty} (-1)^{k_2} \left(\alpha_{0, 2n_2 k_2}-
\sqrt{2}\, \alpha_{2n_1,2 n_2 k_2}  \right),
\end{aligned}
\end{equation*}
and
\begin{equation*}
\begin{aligned}
\tilde{\cE}^{(2)}_{n_1,n_2}& = \sqrt{2} \alpha_{2n_1,2n_2}+
\sum_{k_1=2}^{\infty}  \left(\alpha_{2n_1 k_1,0}+ \sqrt{2}\, \alpha_{2n_1 k_1,2
n_2}  \right) \\
\phantom{=} &+ \sqrt{2} \sum_{k_1=2}^{\infty}\sum_{k_2=2}^{\infty}
\alpha_{2 n_1 k_1, 2 n_2 k_2} + \sum_{k_2=2}^{\infty}  \left(\alpha_{0, 2n_2
k_2}+ \sqrt{2}\, \alpha_{2n_1,2 n_2 k_2}  \right),
\end{aligned}
\end{equation*}
then the statement of Theorem \ref{theorem2} holds true.
\end{corollary}
\begin{proof}
The identity
\begin{equation*}
G_n(p_i^{(\ell)})= \begin{cases}
(-1)^{k} \sqrt{2 \pi}, & \text{if $i=2nk$}, \\
0, & \text{otherwise},
\end{cases}
\end{equation*}
reported in proof of Corollary 1 in \cite{DFR2020}, allows us to
obtain a simplified expression for the terms $S_{n_1,n_2}$, $\cE^{(1)}_{n_1,n_2}$ and $\cE^{(2)}_{n_1,n_2}$ given in Theorem \ref{theorem2}, that is,
\begin{equation*}
S_{n_1,n_2} = \sqrt{2} \pi (\alpha_{2n_1,0}+\alpha_{0,2n_2}), \quad \cE^{(1)}_{n_1,n_2} = \sqrt{2} \pi \tilde{\cE}^{(1)}_{n_1,n_2}, \quad \cE^{(2)}_{n_1,n_2} = \sqrt{2} \pi \tilde{\cE}^{(2)}_{n_1,n_2}.
\end{equation*}
By applying Theorem \ref{theorem2}, we conclude the proof.
\end{proof}

We remark here that, for the Chebychev case, the number of coefficients
$\alpha_{i,j}$ present in the different series terms is much smaller than the
ones involved in the completed expression of $|\cE^{(1)}_{n_1,n_2}|$ and
$|\cE^{(2)}_{n_1,n_2}|$ introduced in Theorem~\ref{theorem2} proof,
simplifying the expression \eqref{condsign}.

In the next corollary, we further streamline the results in Corollary \ref{corollary1}.

\begin{corollary}\label{corollary2}
Let us consider $\alpha_i=\beta_i=-\frac{1}{2}$ in \eqref{w}. Then, if
$$
|\theta_{n_1,n_2}|< |\alpha_{2n_1,0}+\alpha_{0,2n_2} |,
$$
holds true for $n_1$ and $n_2$ large enough, where
\begin{align*}
|\theta_{n_1,n_2}|&= \sqrt{2}|\alpha_{2n_1,2n_2}|+\sum_{k_1=2}^{\infty}  |\alpha_{2n_1 k_1,0}|+ \sqrt{2}\, |\alpha_{2n_1 k_1,2 n_2}|\\ &+\sqrt{2} \sum_{k_1=2}^{\infty}\sum_{k_2=2}^{\infty} |\alpha_{2 n_1 k_1, 2 n_2 k_2}| + \sum_{k_2=2}^{\infty} |\alpha_{0, 2n_2 k_2}|+ \sqrt{2}\, |\alpha_{2n_1,2 n_2 k_2}|,
\end{align*}
then the statement of Theorem \ref{theorem2} holds true.
\end{corollary}

\begin{proof}
By using the triangle inequality and taking into account the hypothesis, we have
\begin{equation*}
\max(|\tilde{\cE}^{(1)}_{n_1,n_2}|,|\tilde{\cE}^{(2)}_{n_1,n_2}|) \leq |\theta_{n_1,n_2}| \leq |\alpha_{2n_1,0}+\alpha_{0,2n_2} |,
\end{equation*}
which yields the assertion, by virtue of Theorem \ref{theorem2}.
\end{proof}

\section{Nystr\"om methods and the averaged Nystr\"om interpolant}\label{sec:nystrom}

The aim of this section is to approximate the solution of \eqref{Fredholm}
by an interpolant function whose construction is based on Gauss and
anti-Gauss cubature rules \eqref{Gn} and \eqref{A}.

If the right hand side in equation \eqref{Fredholm} has a low regularity
at $\pm 1$, the solution inherits the same smoothness.
The same happens if the kernel exhibits a similar behaviour at $\pm 1$ with
respect to the external variable $\bm{y}$.
Therefore, we solve the equation in a suitable weighted space.
Let us introduce the weight function
\begin{equation}\label{u}
u(\xx) =u_1(x_1) u_2(x_2),
\end{equation}
with $u_i(x_i)=(1-x_i)^{\gamma_i} (1+x_i)^{\delta_i}$ for $\gamma_i,\delta_i \geq 0$ and $i=1,2.$
We search for the solution of \eqref{Fredholm} in the space $C_u$ of all
functions $f$ continuous
in the interior of the square $\cS$ and such that
\begin{align*}
\begin{cases}
\displaystyle \lim_{x_1 \to \pm 1} (fu)(x_1,x_2)  =0, \qquad  \forall x_2 \in
[-1,1], \\
\displaystyle \lim_{x_2 \to \pm 1} (fu)(x_1,x_2)=0, \qquad  \forall x_1 \in
[-1,1],
\end{cases}
\end{align*}
endowed with the norm
$$
\|f\|_{C_u}=\|fu\|_\infty=\sup_{\xx\in\cS} |(fu)(\xx)|.
$$
If $\gamma_i=\delta_i=0$ for $i=1,2$, then $C_u$ coincides with the set
of all continuous functions on the square, i.e., $C_u \equiv C(\cS)$.
If any partial derivative of the function $f$ has one or more
singularities at the boundary of
$\cS$, then the corresponding parameter $\gamma_i$ or $\delta_i$ is set
to a positive value in order to compensate the singularity.

This approach amounts to solving the weighted equation
\begin{equation}\label{weq}
(fu)(\yy) - \int_{\cS}  k(\xx,\yy) \frac{u(\yy)}{u(\xx)}
(fu)(\xx) w(\xx) d\xx = (gu)(\yy),
\end{equation}
in the space $C(\cS)$ of continuous functions on the square.

To deal with smoother functions having some discontinuous derivatives on the
boundary of $\cS$, we introduce the Sobolev-type space
$$
W^r_u=\{f\in C_u: \|f_{x_i}^{(r)}\varphi^{r}u\|_\infty<\infty,\ i=1,2\},
$$
where $\varphi(z)=\sqrt{1-z^2}$. The superscript $(r)$ denotes the $r$th
derivative of the univariate function $f_{x_i}$, obtained by fixing either
$x_1$ or $x_2$ in the function $f$.
We equip $W^r_u$ with the norm
$$
\|f\|_{W^r_u}=\|fu\|_\infty+\max_{i=1,2}\|f_{x_i}^{(r)}\varphi^{r}u\|_\infty.
$$

The error of best polynomial approximation in $C_u$ can be defined as
$$
E_{m,n}(f)_u=\inf_{p\in\mathbb{P}_{m,n}}\|[f-p]u\|_\infty.
$$

From now on, the symbol $\cC$ will denote a positive constant and we
will use the notation $\cC \neq \cC(a,b,\ldots)$ to say that
$\cC$ is independent of the parameters $a,b,\ldots$, and
$\cC=\cC(a,b,\ldots)$ to say that it depends on them.
Moreover, if $A,B>0$ are quantities depending on some parameters, we will write
$A \sim B$, if there exists a positive constant $\cC \neq
\cC(A,B)$ such that $\frac{B}{\cC} \leq A \leq C B.$

Next proposition gives an estimate for the above error in Sobolev-type spaces.

\begin{proposition}
For each $f\in W^r_u$, it holds
$$
E_{m,n}(f)_u \leq \cC\left[ \frac{1}{m^r} + \frac{1}{n^r} \right]
\cdot \max_{i=1,2}\|f_{x_i}^{(r)}\varphi^{r}u\|_\infty,
$$
where $\cC \neq \cC(m,n,f)$.
\end{proposition}

\begin{proof}
Following \cite[Theorem 2.1]{OccorsioRusso2018}, one has
\begin{equation*}
E_{m,n}(f)_u \leq \cC \left[\sup_{x_2 \in [-1,1]} u_2(x_2) E_m(f_{x_2})_{u_1}+ \sup_{x_1 \in [-1,1]} u_1(x_1) E_n(f_{x_1})_{u_2} \right],
\end{equation*}
where $E_\ell(g)_{u_i}$ is the $u_i$-weighted best approximation error of the
univariate function $g$ by a polynomial of degree at most $\ell$; see
\cite[estimate (2.5.16)]{MMlibro}.
Then, by the inequality
$$
E_\ell(g)_{u_i} \leq \frac{\cC}{\ell^r}
\|g^{(r)}\varphi^{r}u_i\|_\infty,
$$
from \cite[estimate (2.5.22)]{MMlibro}, we obtain the assertion.
\end{proof}

To ease the exposition, we introduce a multi-index notation, where an index may
take integer vectorial values. Such indexes will be denoted by bold letters.
Let $\nn=(n_1,n_2)$ and consider the set of bi-indices
\begin{equation*}
\mathfrak{I}_\nn=\left\{\ii=(i_1,i_2) :
1\leq i_1\leq n_1,\, 1\leq i_2\leq n_2\right\}.
\end{equation*}
For $\ii\in\mathfrak{I}_\nn$, consistently with the notation
$\xx=(x_1,x_2)$, we define $\xx_\ii=(x_{i_1}^{(1)},x_{i_2}^{(2)})$, where
$x_{i_1}^{(1)}$ and $x_{i_2}^{(2)}$ are the Gaussian nodes introduced in the
cubature rule \eqref{Gn}, which we will denote by $\cG_\nn$.

Let us now write the classical Nystr\"om method for the integral equation
\eqref{Fredholm}, based on approximating the operator $K$ by the Gauss cubature
formula $\cG_\nn$.
This leads to the functional equation
\begin{equation}\label{eqfun}
(I-K_\nn)f_\nn=g,
\end{equation}
where $f_\nn$ is an unknown function approximating $f$ and
$$
(K_\nn f)(\yy) = \sum_{j_1=1}^{n_1}\sum_{j_2=1}^{n_2}
\lambda_{j_1}^{(1)} \lambda_{j_2}^{(2)}
\, k(\xx_\jj,\yy) f(\xx_\jj),
$$
where $\jj=(j_1,j_2)\in\mathfrak{I}_\nn$.

By multiplying both sides of \eqref{eqfun} by the weight function $u$ and
collocating at the points $\xx_\ii$, $\ii\in\mathfrak{I}_\nn$, we obtain the
linear system
\begin{equation}\label{system}
\sum_{j_1=1}^{n_1}\sum_{j_2=1}^{n_2}
\left[\delta_{i_1,j_1}\delta_{i_2,j_2}-\lambda_{j_1}^{(1)} \lambda_{j_2}^{(2)}
\, \dfrac{u(\xx_\ii)}{u(\xx_\jj)} \, k(\xx_\jj,\xx_\ii)\right] \,
a_{j_1,j_2}= (gu)(\xx_\ii),
\end{equation}
where $\delta_{i,k}$ is the Kronecker
symbol, and $a_{j_1,j_2}=(fu)(\xx_\jj)$ are the unknowns.
By defining $\delta_{\ii,\jj}=\delta_{i_1,j_1}\delta_{i_2,j_2}$,
$\lambda_\jj=\lambda_{j_1}^{(1)} \lambda_{j_2}^{(2)}$, and collapsing the two
summations into a single one, \eqref{system} can be rewritten as
\begin{equation}\label{newsystem}
\sum_{\jj\in\mathfrak{I}_\nn}
\left[\delta_{\ii,\jj}-\lambda_\jj
\, \dfrac{u(\xx_\ii)}{u(\xx_\jj)} \, k(\xx_\jj,\xx_\ii)\right] \,
a_\jj= (gu)(\xx_\ii), \qquad \ii\in\mathfrak{I}_\nn.
\end{equation}
This corresponds to the Nystr\"om method for the weighted equation \eqref{weq}.

We remark that the quantities $k(\xx_\jj,\xx_\ii)$ are
entries $k_{i_1,i_2,j_1,j_2}$ of a fourth order tensor
$\cK\in\R^{I_1\times I_2\times I_1 \times I_2}$, where
$I_k=\{1,2,\ldots,n_k\}$, $k=1,2$; see~\cite{kolda2009}.
Moreover, the tensor-matrix product in \eqref{newsystem} and the
tensor-tensor product that will be used in next section corresponds to
the so-called Einstein product \cite{brazell2013,einstein1916}.
We prefer to adopt the multi-index formalism, used, e.g., in
\cite{mrs03,mrs06a,mrs06b}, because it is closer to the usual matrix
notation.

The solution of system \eqref{newsystem} provides the unique solution of
equation \eqref{eqfun} and vice-versa. In fact, if $a^*_\jj$ is a solution
of \eqref{newsystem}, then we can determine the weighted solution of
\eqref{eqfun} by the so-called Nystr\"om interpolant
\begin{equation}\label{interpolantG}
(f_\nn u)(\xx)=(gu)(\xx)+ u(\xx)
\sum_{\jj\in\mathfrak{I}_\nn} \frac{\lambda_\jj}{u(\xx_\jj)} \,
k(\xx_\jj,\xx)  \, a^{*}_\jj.
\end{equation}
Vice-versa, if we evaluate \eqref{interpolantG} at the cubature points we
obtain the solution of \eqref{newsystem}.

Now, we apply the Nystr\"om method to the anti-Gaussian cubature
formula $\cA_{\nn+\bm{1}}$, with $\bm{1}=(1,1)$, as an approximation
for the operator $K$, obtaining the equation
\begin{equation}\label{eqfunanti}
(I-\widetilde{K}_{\nn+\bm{1}})\tilde{f}_{\nn+\bm{1}}=g,
\end{equation}
where $\tilde{f}_{\nn+\bm{1}}$ is the unknown and
$$
(\widetilde{K}_{\nn+\bm{1}} f)(\yy) = \sum_{\jj\in\mathfrak{I}_{\nn+\bm{1}}}
\mu_\jj \, k(\etab_\jj,\yy) f(\etab_\jj),
$$
with $\mu_\jj=\mu_{j_1}^{(1)}\mu_{j_2}^{(2)}$ and
$\etab_\jj=(\eta_{j_1}^{(1)},\eta_{j_2}^{(2)})$.

Collocating equation \eqref{eqfunanti} at the knots $\etab_\ii$
and a multiplication of both sides by $u(\etab_\ii)$ leads to the linear system
\begin{equation}\label{systemanti}
\sum_{\jj\in\mathfrak{I}_{\nn+\bm{1}}}
\left[\delta_{\ii,\jj}-\mu_\jj
\, \dfrac{u(\etab_\ii)}{u(\etab_\jj)} \, k(\etab_\jj,\etab_\ii)\right] \,
\tilde{a}_\jj= (gu)(\etab_\ii), \qquad \ii\in\mathfrak{I}_{\nn+\bm{1}},
\end{equation}
where $\tilde{a}_\jj=(fu)(\etab_\jj)$ are the unknowns.

If $\tilde{a}^*_\jj$ is the
solution of \eqref{systemanti}, then the Nystr\"om interpolant
\begin{equation}\label{interpolantA}
(\tilde{f}_{\nn+\bm{1}} u)(\xx)=(gu)(\xx)+ u(\xx)
\sum_{\jj\in\mathfrak{I}_{\nn+\bm{1}}} \frac{\mu_\jj}{u(\etab_\jj)} \,
k(\etab_\jj,\xx)  \, \tilde{a}^{*}_\jj,
\end{equation}
solves \eqref{eqfunanti}, and hence approximates the solution of
\eqref{Fredholm}.
Vice-versa, if we evaluate the above function at the cubature points we
obtain the solution of \eqref{systemanti}.

\begin{theorem}\label{teo:convergenza}
Let $\ker \{I+K\}=\{0\}$ in $C_u$ and let the parameters of the weight
$u$ given in \eqref{u} be such that
$$0 \leq \gamma_i < \alpha_i+1, \qquad 0 \leq \beta_i < \delta_i+1,  \qquad i=1,2. $$
We also assume that
\begin{equation*}
g \in W_u^r, \quad \sup_{\xx \in \cS} \|k_{\xx}\|_{W^r_u} < \infty, \quad
\sup_{\yy \in \cS} u(\yy) \|k_{\yy}\|_{W^r} < \infty.
\end{equation*}
Then, there exist a sufficiently large bi-index $\nn_0$ such that, for
$\nn\geq\nn_0$, equations \eqref{eqfun} and \eqref{eqfunanti} admit a
unique solution $f^*_\nn \in C_u$ and $\tilde{f}^*_{\nn+\bm{1}} \in C_u$,
respectively.
Moreover, if $f^*$ is the unique solution of \eqref{Fredholm}, then
\begin{equation}\label{error}
\max \left\{ \|(f^*-f^*_\nn)u\|_\infty,
\|(f^*-\tilde{f}^*_{\nn+\bm{1}})u\|_\infty \right\}
\leq \cC \left[ \frac{1}{n_1^r} + \frac{1}{n_2^r} \right] \cdot \max_{i=1,2}\|f_{x_i}^{*(r)}\varphi^{r}u\|_\infty,
\end{equation}
where $\cC\neq \cC(\nn,f).$
\end{theorem}
\begin{proof}
The stability of the Nystr\"om method based on the Gauss rule as well as
the error estimate \eqref{error} has been proved in
\cite{OccorsioRusso2011} (see also \cite[Theorem~4.1]{Laguardia2022} for
the case $u \equiv 1$).
The same line of the theorem in~\cite{OccorsioRusso2011} can be followed to
prove the assertion related to the Nystr\"om method concerning the anti-Gauss
rule; see also \cite[Theorem~3.1]{FermoRusso}. 
\end{proof}
\begin{corollary} \label{cor:coeff}
Let $f^*$ be the unique solution of \eqref{Fredholm}.
Consider the orthogonal expansion of the kernel $k$ multiplied by $f^*$ and
its approximations $f_\nn$ and $\tilde{f}_{\nn+\bm{1}}$
$$
\begin{aligned}
k(\xx,\yy) f^*(\xx) &= \sum_{i=0}^\infty \sum_{j=0}^\infty
\alpha_{i,j}(\yy) p_{i,j}(\xx), \,
& \alpha_{i,j}(\yy) &= \left(b_0^{(1)}
b_0^{(2)}\right)^{-\frac{1}{2}}(K(f^*p_{i,j}))(\yy), \\
k(\xx,\yy) f_\nn(\xx) &= \sum_{i=0}^\infty \sum_{j=0}^\infty
\alpha_{i,j}^\nn(\yy) p_{i,j}(\xx), \,
& \alpha_{i,j}^\nn(\yy) &= \left(b_0^{(1)}
b_0^{(2)}\right)^{-\frac{1}{2}}(K(f_\nn p_{i,j}))(\yy), \\
k(\xx,\yy) \tilde{f}_{\nn+\bm{1}}(\xx) &= \sum_{i=0}^\infty \sum_{j=0}^\infty
\tilde{\alpha}_{i,j}^{\nn+\bm{1}}(\yy) p_{i,j}(\xx), \hspace{-0.1cm}
& \tilde{\alpha}_{i,j}^{\nn+\bm{1}}(\yy) &= \left(b_0^{(1)}
b_0^{(2)}\right)^{-\frac{1}{2}}(K(\tilde{f}_{\nn+\bm{1}} p_{i,j}))(\yy).
\end{aligned}
$$
Then, under the same assumption of Theorem~\ref{teo:convergenza},
$$
\lim_{n_1,n_2\to\infty} \| [\alpha_{i,j}^\nn-\alpha_{i,j}]u \|_\infty = 0
\qquad\text{and}\qquad
\lim_{n_1,n_2\to\infty} \| [\tilde{\alpha}_{i,j}^{\nn+\bm{1}}-\alpha_{i,j}]u
	\|_\infty = 0.
$$
\end{corollary}

\begin{proof}
The proof follows the same line of Theorem~4 from \cite{DFR2020}.
\end{proof}

\begin{theorem}\label{teo:braketsol}
Let us assume that inequality \eqref{condsign} is satisfied and the
assumptions of Theorem~\ref{teo:convergenza} are verified.
Then, for any $\yy\in\cS$, either
\begin{equation*}
\tilde{f}_{\nn+\bm{1}}(\yy) \leq f^*(\yy) \leq f_\nn(\yy) \quad \text{or}
\quad f_\nn(\yy) \leq f^*(\yy)  \leq \tilde{f}_{\nn+\bm{1}}(\yy).
\end{equation*}
\end{theorem}

\begin{proof}
By \eqref{Fredholm}, $f=Kf+g$. Proceeding similarly with equations
\eqref{eqfun} and \eqref{eqfunanti}, we deduce that to prove the assertion
it is sufficient to state either of the following two relations
$$
\begin{gathered}
(\widetilde{K}_{\nn+\bm{1}}\tilde{f}_{\nn+\bm{1}})(\yy) \leq (Kf^*)(\yy)
\leq (K_\nn f_\nn)(\yy) \\
(K_\nn f_\nn)(\yy) \leq (Kf^*)(\yy) \leq
(\widetilde{K}_{\nn+\bm{1}}\tilde{f}_{\nn+\bm{1}})(\yy).
\end{gathered}
$$
By virtue of the assumptions and Corollary~\ref{cor:coeff}, the above
inequalities follow by applying Theorem~\ref{theorem2} to the function
$h_\yy(\xx)=k(\xx,\yy)f(\xx)$.
\end{proof}

Once we have proven under which conditions the unique solution $f^*$ of
the integral equation is bracketed by the Nystr\"om interpolants for
any $\yy\in\cS$, we can introduce the averaged Nystr\"om interpolant
\begin{equation}\label{averaged}
\mathfrak{f_\nn(\yy)} = \dfrac{1}{2}\left(f_\nn(\yy)+\tilde{f}_{\nn+\bm{1}}(\yy)\right), \quad \yy\in\cS,
\end{equation}
which yields a more accurate approximated solution.

\section{Solving the linear systems}\label{sec:system}

In this section we describe a tensor representation of systems
\eqref{newsystem} and \eqref{systemanti}, we study their condition number,
and propose numerical methods for their resolution.
In the following, the product between two tensors $\cM$, $\cN$, and
between a tensor $\cM$ and a matrix $\bm{a}$, must be considered in the
multi-index sense, that is, 
$$
(\cM\cN)_{\ii,\jj} = \sum_{\kk\in\mathfrak{I}_\nn} \cM_{\ii,\kk} \cN_{\kk,\jj}, \qquad
(\cM\bm{a})_\ii = \sum_{\kk\in\mathfrak{I}_\nn} \cM_{\ii,\kk} \bm{a}_\kk,
\qquad \ii,\jj\in\mathfrak{I}_\nn.
$$
The inverse tensor is such that $\cM\cM^{-1}=\cI$, where
$(\cI)_{\ii,\jj}=\delta_{\ii,\jj}$.
Moreover, the infinity norm $\|\cM\|_\infty$ is defined in the usual
operatorial sense, and the condition number is
$\kappa_\infty(\cM)=\|\cM\|_\infty\|\cM^{-1}\|_\infty$.

Let us introduce the notation
\begin{equation*}
\Lambda_\nn=\diag(\lambda_\jj)_{\jj\in\mathfrak{I}_\nn}, \quad\text{with}\quad
(\Lambda_\nn)_{\ii,\jj} = \begin{cases}
\lambda_\jj, & \ii=\jj, \\
0 & \ii\neq\jj.
\end{cases}
\end{equation*}
We give a compact representation of systems \eqref{newsystem} and
\eqref{systemanti},
\begin{align}
(\cI_\nn-\cU_\nn\cK_\nn\cU_\nn^{-1}\Lambda_\nn)\aa
&=\hh, \label{systmatrixG} \\
(\cI_{\nn+\bm{1}}-\widetilde{\cU}_{\nn+\bm{1}}\widetilde{\cK}_{\nn+\bm{1}}
\widetilde{\cU}_{\nn+\bm{1}}^{-1}\widetilde{\Lambda}_{\nn+\bm{1}})
\widetilde{\aa} &= \widetilde{\hh}, \label{systmatrixA}
\end{align}
where $\cK_{\ii\jj} = k(\xx_\jj,\xx_\ii)=k(x_{j_1}^{(1)},x_{j_2}^{(2)},x_{i_1}^{(1)},x_{i_2}^{(2)})$, 
$\cU_\nn = \diag(u(\xx_\jj))_{\jj\in\mathfrak{I}_\nn}$, and
$\hh = ((gu)(\xx_\ii))_{\ii\in\mathfrak{I}_\nn}$.
Matrices $\widetilde{\cU}_{\nn+\bm{1}}$,
$\widetilde{\cK}_{\nn+\bm{1}}$, $\widetilde{\Lambda}_{\nn+\bm{1}}$, and the
array $\widetilde{\hh}$ are defined similarly.

In the next theorem we state the numerical stability of the Nystr\"om method.

\begin{theorem}\label{teo:cond}
Under the assumptions of Theorem \ref{teo:convergenza}, it holds
$$
\kappa_\infty(\cI_\nn -
\cU_\nn\cK_\nn\cU_\nn^{-1}\Lambda_\nn)
\leq \cC, \qquad
\kappa_\infty(\cI_{\nn+\bm{1}} - \cU_{\nn+\bm{1}}\widetilde{\cK}_{\nn+\bm{1}}
\widetilde{\cU}_{\nn+\bm{1}}^{-1}\widetilde\Lambda_{\nn+\bm{1}})
\leq \cC,
$$
where $\cC$ is independent of $\nn$.
\end{theorem}

\begin{proof}
The proof follows the same idea of Theorem~3.1 from \cite{OccorsioRusso2011}.
\end{proof}

\subsection{The general case}

Let us first solve linear systems \eqref{systmatrixG} and
\eqref{systmatrixA} in the general case, that is, when
the coefficient tensor is not structured.
For the sake of clarity and brevity, from now on we will only refer to
system \eqref{systmatrixG} and set
$
\cF_\nn=\cI_\nn-\cU_\nn\cK_\nn\cU_\nn^{-1}\Lambda_\nn.
$
The same considerations will be valid for system \eqref{systmatrixA}
and the corresponding tensor $\widetilde{\cF}_{\nn+\bm{1}}$.
We note that even if the kernel is a symmetric function like, for instance,
$k(\xx,\yy) = \xx^2+\yy^2+\xx \yy$, the resulting coefficient tensor may be
not symmetric, that is, $(\cF_\nn)_{\ii,\jj}\neq(\cF_\nn)_{\jj,\ii}$, due to
the presence of the weight function $u$ and the Christoffel numbers. 

Before solving system \eqref{systmatrixG}, we rewrite it in
matrix form, i.e., we transform the matrices containing the unknowns and
the right-hand side into vectors, and represent the
multi-index tensor as a standard matrix. To do this, we employ the
lexicographical order to obtain
the matrix  $\bm{F}_N \in \R^{N \times N}$ given by
$$
(\bm{F}_N)_{\ell,k}=(\cF_\nn)_{\ii,\jj}, \quad \ell=i_1+(i_2-1)n_1, 
\quad k=j_1+(j_2-1)n_1.
$$
This process is known as \emph{matricization} or
unfolding~\cite{kolda2009}.
A similar procedure is applied to arrays $\bm{a}$ and $\bm{h}$ to obtain 
vectors $\bar{\bm{a}},\bar{\bm{h}}\in\R^{N}$, with $N=n_1n_2$, defined as
\begin{align*}
\bar{a}_k= a_{j_1,j_2}, \qquad \bar{h}_k= h_{j_1,j_2}, \qquad k=j_1+
(j_2-1)n_1,
\end{align*}
for $j_1=1,\dots,n_1$, $j_2=1,\dots,n_2$, and $k=1,\dots,N$, 
so that the system becomes
\begin{equation}\label{gensystem}
\bm{F}_N\bar{\bm{a}}=\bar{\bm{h}}.
\end{equation}

To solve system \eqref{gensystem}, we employ the generalized
minimal residual (GMRES) method~\cite{SaadSchultz86}.
The GMRES iterative method for the solution of the linear system \eqref{gensystem} is based on the Arnoldi partial
factorization
$\bm{F}_N Q_r=Q_{r+1}H_{r+1,r}$, for  $r=1,2,\ldots,N$,
where $Q_r=[q_1,q_2,\ldots,q_r]$ has orthonormal columns, with
$q_1=\bar{\bm{h}}/\|\bar{\bm{h}}\|$, and $H_{r+1,r}$ is an Hessenberg matrix;
$\|\cdot\|$ denotes the vector 2-norm.

At the $r$th iteration, GMRES approximates the solution of the
system as
$$
\bar{\bm{a}}^{(r)}
=\arg\min_{\bar{\bm{a}}\in K_r}\|\bm{F}_N\bar{\bm{a}}-\bar{\bm{h}}\|^2
=\min_{y\in\mathbb{R}^r}
\|H_{r+1,r} y-\|\bar{\bm{h}}\|\boldsymbol{\e}_1\|^2,
$$
where
$K_r=\sspan\{\bar{\bm{h}},\bm{F}_N\bar{\bm{h}},\ldots,\bm{F}_N^{r-1}\bar{\bm{h}}\}=\sspan\{q_1,\ldots,q_r\}$
is a Krylov space of dimension $r$.

Once the tensor $\cK_\nn$ has been computed,
this requires $2N^2$ floating point operations to assemble the
matrix $\bm{F}_N$ and a matrix-vector product at each iteration, leading
to a computational cost of $O((2+r)N^2)$.

The complexity can be slightly reduced by avoiding to assemble $\bm{F}_N$
and performing the product $\bm{F}_N\bm{q}_k$ at each iteration as
$
\bm{q}_k-\bm{u}\circ[K_N(\bm{d}\circ\bm{q}_k)],
$
where $\bm{u}_\jj=u(\bm{x}_\jj)$, $\bm{d}_\jj=\lambda_\jj/\bm{u}_\jj$, 
$K_N$ is the matricization of $\cK_\nn$, and $\circ$ denotes the
componentwise Hadamard product
$(\bm{a}\circ\bm{b})_\jj=\bm{a}_\jj\bm{b}_\jj$.
In this case the computational cost is $O(rN^2)$.
We will denote this approach with a factored coefficient matrix by GMRES-FM.

\subsection{The case of a separable kernel}

Let us assume that the kernel in \eqref{Fredholm} is separable, that is, $k(\xx,\yy) = k(x_1,x_2,y_1,y_2) = k_1(x_1,y_1) \, k_2(x_2,y_2)$.
This means that $\cK_\nn=K^{(1)}_{n_1}\otimes K^{(2)}_{n_2}$, where
$K^{(1)}_{n_1}$ and $K^{(2)}_{n_2}$ are two square matrices of dimension $n_1$
and $n_2$, respectively, with $(K^{(1)}_{n_1})_{i_1,j_1} = k_1(x_{j_1}^{(1)},x_{i_1}^{(1)})$ and $(K^{(2)}_{n_2})_{i_2,j_2} = k_2(x_{j_2}^{(2)},x_{i_2}^{(2)}),
$
and $\otimes$ denotes the Kronecker tensor product, that is,
$
(\cK_\nn)_{\ii,\jj} = (K^{(1)}_{n_1})_{i_1,j_1} 
(K^{(2)}_{n_2})_{i_2,j_2}.
$

Keeping into account that $u(\xx)=u_1(x_1) u_2(x_2)$ and
$\lambda_\jj=\lambda_{j_1}^{(1)}\lambda_{j_2}^{(2)}$, the system
\eqref{newsystem} becomes
\begin{equation*}
\sum_{j_1=1}^{n_1} \sum_{j_2=1}^{n_2}
\left[\delta_{i_1,j_1}\delta_{i_2,j_2} - \phi^{(1)}_{i_1,j_1}
\phi^{(2)}_{i_2,j_2} \right] \,
a_{j_1,j_2} = h_{i_1,i_2},
\end{equation*}
for $i_1=1,\ldots,n_1$ and $i_2=1,\ldots,n_2$,
with
\begin{equation*}
\phi^{(\ell)}_{i_\ell,j_\ell} =
\lambda^{(\ell)}_{j_\ell} \, 
\frac{u_\ell(x^{(\ell)}_{i_\ell})}{u_\ell(x^{(\ell)}_{j_\ell})}
\, (K^{(\ell)}_{n_\ell})_{i_\ell,j_\ell}, \qquad \ell=1,2.
\end{equation*}
This amounts to solving the Stein matrix equation
\begin{equation}\label{stein}
\Phi^{(1)} A (\Phi^{(2)})^T - A + H = 0,
\end{equation}
where $A=(a_{j_1,j_2})$, $H=(h_{i_1,i_2})$, and
$\Phi^{(\ell)}=(\phi^{(\ell)}_{i_\ell,j_\ell})$, for $\ell=1,2$.
There is a wide literature on numerical methods for solving this kind of
matrix equations, some classical references
are~\cite{barr1977,golub1979,hamm1982}.
We will use the \texttt{dlyap} function of MATLAB.

The structure of the Stein equation \eqref{stein} also allows for speeding
up the GMRES method and reducing the storage space.
Indeed, the product $\bm{F}_N\bm{q}_k$ can be expressed, at each
iteration, in the form
$
Q_k - \Phi^{(1)} Q_k (\Phi^{(2)})^T,
$
where the vector $\bm{q}_k$ is the unfolding of the matrix $Q_k$.
In this way, the number of floating point operations of a matrix-vector
product decreases from $O(N^2)$ to $O(N)$, as well as the storage space.
This implementation will be denoted in the following by GMRES-SK.

\section{Numerical results}\label{sec:tests}

In this section, we investigate the performance of the numerical methods
described through the paper. We analyze separately the approximation of
bivariate integrals and the numerical solution of Fredholm integral equations
of the second-kind.

\subsection{Approximation of integrals}

In the next two examples, we give a practical illustration of
the theoretical properties of the cubature rules presented in 
Section~\ref{sec:cubature}.
In both cases, the exact value $I(f)$ of the integral is not available.
We consider as exact the value $\cG_{512,512}(f)$, produced by the Gauss
cubature rule \eqref{Gn} when $n_1=n_2=512$.
The tables display the cubature errors 
$$
R^{(G)\strut}_{n_1,n_2}(f),\ R^{(A)}_{n_1,n_2}(f), \text{ and }
R^{[1]}_{n_1,n_2}(f),
$$
defined in equations \eqref{erreg}, \eqref{errea}, and
\eqref{erre1}, respectively.
In addition, 
we also report 
$$
R^{(Avg)}_{n_1,n_2}(f)=\cI(f)-\cG^{Avg}_{2n_1+1,2n_2+1}(f).
$$

\begin{example}\label{example1}
Let us consider the following integral
\begin{equation}\label{eq:example1}
\int_{-1}^1 \int_{-1}^1 |\sin(1-x_1)|^{\frac{9}{2}} (1+x_1+x_2) w(x_1,x_2) dx_1 dx_2,
\end{equation}
where $w$ is the weight function defined in \eqref{w} with $\alpha_1=\beta_1=-1/2$ and $\alpha_2=\beta_2=0$.
The integrand function is smooth with respect to the variable $x_2$, whereas
only its first four derivatives with respect to $x_1$ are continuous.
Hence, it is sufficient to use few points (for instance $n_2=8$) to approximate
the integral in $x_2$.
In Table \ref{tab:example1} we report the cubature errors for increasing values
of $n_1$. 
From the third and fourth columns, we can see that the error provided by the
anti-Gauss rule is of the same magnitude of the error given by the Gauss rule
and opposite in sign.
This improves the accuracy of the averaged rule; see the fifth column. The
sixth column of the table demonstrates that formula
$\cG^{Avg}_{2n_1+1,2n_2+1}(f)$ is a good estimate for the Gauss rule error.

\begin{table}[ht]
\caption{Cubature errors for Example \ref{example1}.}
\label{tab:example1}
\begin{center}
\begin{tabular}{c c c c c c c c}
\hline
$n_1$ & $n_2$ & $R^{(G)\strut}_{n_1,n_2}(f)$ & $R^{(A)}_{n_1,n_2}(f)$ 
	& $R^{(Avg)}_{n_1,n_2}(f)$ & $R^{[1]}_{n_1,n_2}(f)$
	& $R^{(\text{KX$_2$})}_{n_1,n_2}(f)$ & $R^{(\text{KX$_3$})}_{n_1,n_2}(f)$ \\
\hline
4 	& 8 	& 1.63e-03 	& -1.63e-03 	& 1.27e-07 	& 1.63e-03  & 2.90e-03 	& 1.85e-03 \\
8 	& 8 	& -1.27e-07 	& 1.27e-07 	& 1.22e-10 	& -1.27e-07 & 4.15e-04 	& 9.83e-04 \\ 
16 	& 8 	& -1.21e-10 	& 1.22e-10 	& 1.11e-13 	& -1.22e-10 & -4.48e-04 & 5.36e-05 \\ 
32 	& 8 	& -1.15e-13 	& 1.10e-13 	& -2.66e-15 	& -1.12e-13 & -1.04e-04 & 2.14e-06 \\ 
64 	& 8 	& -2.22e-15 	& -3.11e-15 	& -2.66e-15 	& 4.44e-16  & -9.23e-06 & - 	   \\
\hline
\end{tabular}
\end{center}
\end{table}


The graph on the left in Figure~\ref{fig:alfabeta} displays the two terms of
inequality \eqref{condsign} for $n_1=1,\ldots,30$ and $n_2=8$.
It shows that the assumption of Theorem~\ref{theorem2} is numerically verified,
ensuring the change of sign in the errors of the two cubature rules.

\begin{figure}[ht]
\includegraphics[width=0.48\textwidth]{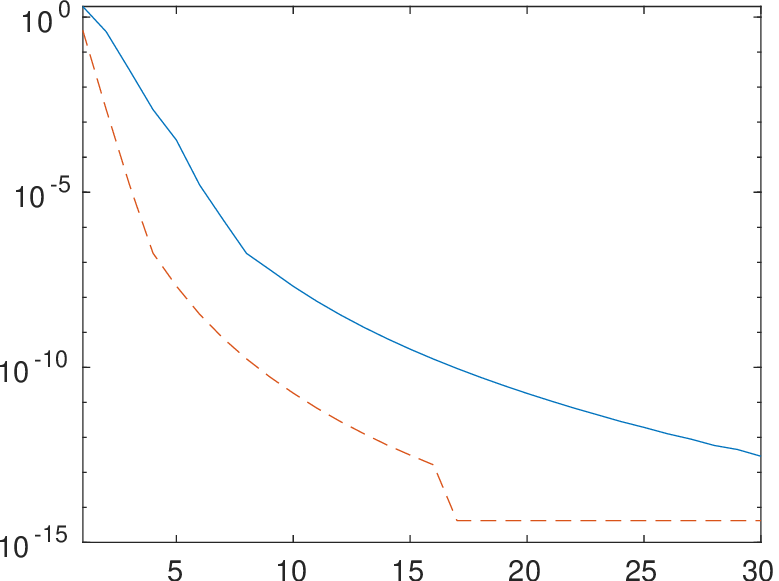}
\includegraphics[width=0.48\textwidth]{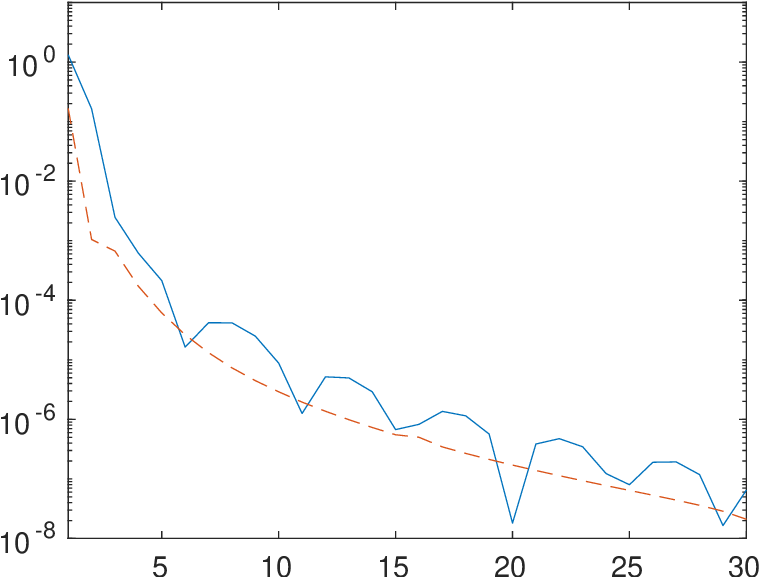}
\caption{Graph of the two terms in inequality \eqref{condsign} for
Example~\ref{example1} (left) and Example~\ref{example2} (right):
$|S_{n_1,n_2}|$ is represented by a continuous line,
$\max(|\cE^{(1)}_{n_1,n_2}|,|\cE^{(2)}_{n_1,n_2}|)$ by a dashed one.}
\label{fig:alfabeta}
\end{figure}

In the last two columns of Table~\ref{tab:example1}, we compare
our results with those obtained using the Gauss-Legendre-type quadrature rule
introduced in \cite{Kaneko1994} for weakly singular integrals.
The method focuses on evaluating the integral
$$I(f)=\int_0^1 f(x) dx,$$
where $f$ belongs to a class denoted by $\text{Type}(\alpha,k,S)$ containing
function with specific singularities at a finite number of points; we refer to
\cite{Kaneko1994} for more details about its definition.
The authors divide the integral into $m$ subintegrals, each one defined on
$[t_j,t_{j+1}]$, $j=0,\ldots,n$, constituting a partition of $[0,1]$.
A mapping is then introduced to transform $[t_j,t_{j+1}]$ to $[-1,1]$,
and the resulting integral is then approximated using a Gauss-Legendre formula
with $k$ nodes, yielding
\begin{equation}\label{eq:xu}
I_{n,k}(f)=\sum_{j=1}^{n-1} \frac{t_{j+1}-t_j}{2} 
\sum_{i=1}^k \lambda_i f\left(\frac{t_{i+1}-t_i}{2}u_i+\frac{t_{i+1}+t_i}{2}
\right),
\end{equation}
where $\lambda_i$ are the weights of the quadrature formula and $u_i$ are the zeros of the Legendre polynomial of degree $k$.
We implemented the algorithm and tested it on some of the examples reported
in~\cite{Kaneko1994}, reaching at least the same accuracy.

Successively, we applied the above formula to integral \eqref{eq:example1} by a
tensor product approach similar to the one used in our methods.
We note that the integrand function has singularities at $x_1=\pm 1$.
In order to follow the approach in \cite{Kaneko1994}, we first divided the
integral in $x_1$ into two parts, i.e.,
\begin{align*}
\int_{-1}^1&\left[ \int_{-1}^0 |\sin(1-x_1)|^{\frac{9}{2}} (1+x_1+x_2) \frac{1}{\sqrt{1-x_1^2}} dx_1 \right.\\ &+ \left. \int_{0}^1 |\sin(1-x_1)|^{\frac{9}{2}} (1+x_1+x_2) \frac{1}{\sqrt{1-x_1^2}} dx_1 \right] dx_2,
\end{align*}
separating the singularities.
The two integrand functions are of Type($-\frac{1}{2}$,$k$,$\{-1\})$ and
Type($-\frac{1}{2}$,$k$,$\{1\}$), respectively, according to the terminology of
\cite{Kaneko1994}.

As the integrand function is smooth with respect to the second variable,
we fixed $n_2=8$ to approximate the integral in $x_2$, and varied $n_1$ for the
integral in $x_1$.
The last two columns in Table \ref{tab:example1} presents the errors 
$R^{(\text{KX$_2$})}_{n_1,n_2}(f)$ and $R^{(\text{KX$_3$})}_{n_1,n_2}(f)$,
with respect to the reference solution $\cG_{512,512}(f)$, computed by setting 
$k=2$ and $k=3$ in \eqref{eq:xu}, respectively.


As expected, the errors obtained by the method based on
\cite{Kaneko1994} decreases as the value of $n_1$ increases.
However, by comparing these results with the other columns in the table, it is
evident that the averaged formula proposed in this paper is significantly more
accurate, achieving an error of the order $10^{-15}$ with just $n_1=32$.
It is important to note that the number of nodes $n_1$ expresses the number of function evaluations for the averaged formula, while formula \eqref{eq:xu} performs $k(n_1-1)$ evaluations.

Finally, we observe that for $n_1=64$ and $k=3$ the method breaks
down because of an overflow, due to one of the nodes $t_j$ becoming too close
to the singularities.
\end{example}

\begin{example}\label{example2}
Let us consider the integral \eqref{int} with
$$
f(x_1,x_2)= x_1 \left|\cos\left(\frac{1}{2}-x_1\right)\right|^{\frac{3}{2}}+x_2 |\sin(1+x_2)|^{\frac{3}{2}},$$
and
$ w(x_1,x_2)=\sqrt{\frac{1-x_1^2}{1-x_2}}.$

In this case, the integrand function has a low smoothness with respect to both
variables. Then, to obtain a good approximation we need to increase both $n_1$
and $n_2$. In Table \ref{tab:example2}, we can see the computational advantage
of the averaged rule with respect to the Gauss scheme. To obtain an order error of $10^{-13}$, we have two choices: we may apply the averaged rule
with $n_1=n_2=128$, and this requires $n_1n_2+(n_1+1)(n_2+1)=33.025$ function
evaluations, or we may use the Gauss cubature formula with $n_1=n_2=256$.
In this case, we have to perform $n_1n_2=65.536$ function evaluations.

\begin{table}[ht]
\caption{Cubature errors for  Example \ref{example2}.}
\label{tab:example2}
\begin{center}
\begin{tabular}{c c c c c c}
\hline
$n_1$ & $n_2$ & $R^{(G)}_{n_1,n_2}(f)$ & $R^{(A)}_{n_1,n_2}(f)$ & $R^{(Avg)}_{n_1,n_2}(f)$ & $R^{[1]}_{n_1,n_2}(f)$ \\
\hline
8 	& 8 	& -1.53e-05 	& 1.55e-05 	& 9.05e-08 	& -1.54e-05 \\
16 	& 16 	& -4.66e-07 	& 4.72e-07 	& 2.98e-09 	& -4.69e-07 \\
32 	& 32 	& -1.49e-08 	& 1.51e-08 	& 9.62e-11 	& -1.50e-08 \\
64 	& 64 	& -4.73e-10 	& 4.79e-10 	& 3.07e-12 	& -4.76e-10 \\
128 	& 128 	& -1.49e-11 	& 1.51e-11 	& 1.13e-13 	& -1.50e-11 \\
256 	& 256 	& -4.51e-13 	& 4.84e-13 	& 1.60e-14 	& -4.67e-13 \\
\hline
\end{tabular}
\end{center}
\end{table}

The graph on the right in Figure~\ref{fig:alfabeta} shows that for some values
of $n_1=n_2$ the assumption \eqref{condsign} of Theorem \ref{theorem2} is
violated.
However, numerical experiments show that the change of sign in the error
always happens. In particular, the graph shows that inequality
\eqref{condsign} is not verified when $n_1=n_2=20$, but we have
$R^{(G)}_{20,20}(f)=-1.54\cdot 10^{-07}$ and
$R^{(A)}_{20,20}(f)=1.56\cdot 10^{-07}$.
\end{example}

\subsection{Numerical solution of Fredholm integral equations of the second-kind}

In this section we numerically solve several integral equations of the
type \eqref{Fredholm} to investigate the performance of the method
presented in Section \ref{sec:nystrom} and Section \ref{sec:system}, and
support the theoretical analysis.  
For each example, we first identify the space $C_u$, in which the solution
is sought, according to Theorem \ref{teo:convergenza}, solve systems
\eqref{newsystem} and \eqref{systemanti}, compute the Nystr\"om
interpolant \eqref{interpolantG} and \eqref{interpolantA}, and calculate
the averaged Nystr\"om interpolant \eqref{averaged}. 
The algorithms were implemented in Matlab version 9.10 (R2021a), and the
numerical experiments were carried out on an Intel(R) Xeon(R) Gold 6136
server with 128 GB of RAM memory and 32 cores, running the Linux
operating system.

To test the accuracy, we compute the relative errors
\begin{align*}
\xi^{(G)}_{\nn}= \frac{\|f^*-f_{\nn}\|_{C_u}}{\|f^*\|_{C_u}},
\qquad
\xi^{(A)}_{\nn}= \frac{\|f^*-\tilde{f}_{\nn+\bm{1}}\|_{C_u}}{\|f^*\|_{C_u}},
\qquad
\xi^{(Avg)}_{\nn}=\frac{\|f^*-\mathfrak{f_\nn}\|_{C_u}}{\|f^*\|_{C_u}},
\end{align*} 
where the infinity norm is approximated on a grid of $50 \times 50$ points
and $f^*$ is the exact solution of the equation. If the solution is
unknown, then we consider the approximated solution obtained by
the Nystr\"om interpolant \eqref{interpolantG} for sufficiently large
$n_1$, $n_2$ as exact. The adopted value of $\nn=(n_1,n_2)$ will be specified case
by case. 

In our tests, we consider both separable and non-separable kernels.
When a low regularity of the kernel and/or the right-hand side yields the
necessity of increasing the size of the linear system, we explore the
efficiency of the proposed approaches for its solution methods, both in
terms of accuracy and of computational time.
In some examples, we report the $\infty$-norm condition numbers
$\kappa_{\infty}^{(G)}$ and $\kappa_{\infty}^{(A)}$ of systems
\eqref{systmatrixA} and \eqref{systmatrixG}, respectively, to confirm the
theoretical analysis of~Theorem \ref{teo:cond}. 

\begin{example}\label{test1}
Let us first test our method on an integral equation whose exact
solution is known. Consider the equation
\begin{equation*}
f(y_1,y_2)- \int_{-1}^1 \int_{-1}^1 x_2 y_2 \e^{x_1+y_1} f(x_1,x_2) dx_1 dx_2 =
g(y_1,y_2),
\end{equation*}
with right-hand side
$g(y_1,y_2) = \cos{(y_1+y_2)}-(\cos{2}+\e^2(\sin{2}-1))y_2 \e^{y_1-1}$ and
solution $f(x_1,x_2)=\cos{(x_1+x_2)}$. Since the right-hand side and kernel
are smooth functions, we search for the solution in
the space $C_u$ with $u \equiv 1$, i.e., we set $\gamma_i=\delta_i=0$ for
$i=1,2$.

Table~\ref{tab:test1} displays the relative errors for increasing values
of $n_1=n_2$. As expected, since the kernel and right-hand side are
analytic functions, it shows a fast convergence. 
The averaged Nystr\"om interpolant allows to improve accuracy up to four
significant digits, with respect to the two Nystr\"om interpolants based, respectively, on the Gauss and anti-Gauss rules.    
Since the size of the system is small, in this example we solve the
linear systems by Gauss's method with column pivoting.
As highlighted by the last two columns of Table~\ref{tab:test1}, the two
systems are very well conditioned.

\begin{table}[ht]
\caption{Numerical results for Example \ref{test1}.}
\label{tab:test1}
\begin{center}
\begin{tabular}{c c c c c c}
\hline
$\nn$ & $\xi^{(G)}_{\nn}(f)$ & $\xi^{(A)}_{\nn}(f)$ & $\xi^{(Avg)}_{\nn}(f)$ & $\kappa_{\infty}^{(G)}$ & $\kappa_{\infty}^{(A)}$ \\
\hline
(2,2)  & 	 3.79e-02 & 	 3.30e-02 & 	 2.43e-03 & 2.678 & 	 8.504 \\
(4,4)   & 	 2.38e-06 & 	 2.38e-06 & 	 3.00e-10 & 19.016 & 	 30.849 \\
(6,6)   & 	 2.50e-11 & 	 2.50e-11 & 	 1.33e-15 & 30.308 & 	 36.235\\
(8,8)   & 	 5.55e-16 & 	 9.99e-16 & 	 7.22e-16 & 34.967 & 	 34.941\\
\hline
\end{tabular}
\end{center}
\end{table}

A plot of the pointwise errors for the Gauss and the anti-Gauss
interpolants is reported in Figure \ref{fig:test1}, for $\nn=(4,4)$, in
two different perspectives. It can be observed that the errors provided by the two cubature rules are of opposite sign, confirming the assertion of Theorem~\ref{teo:braketsol}.

\begin{figure}[ht]
\includegraphics[width=0.49\textwidth]{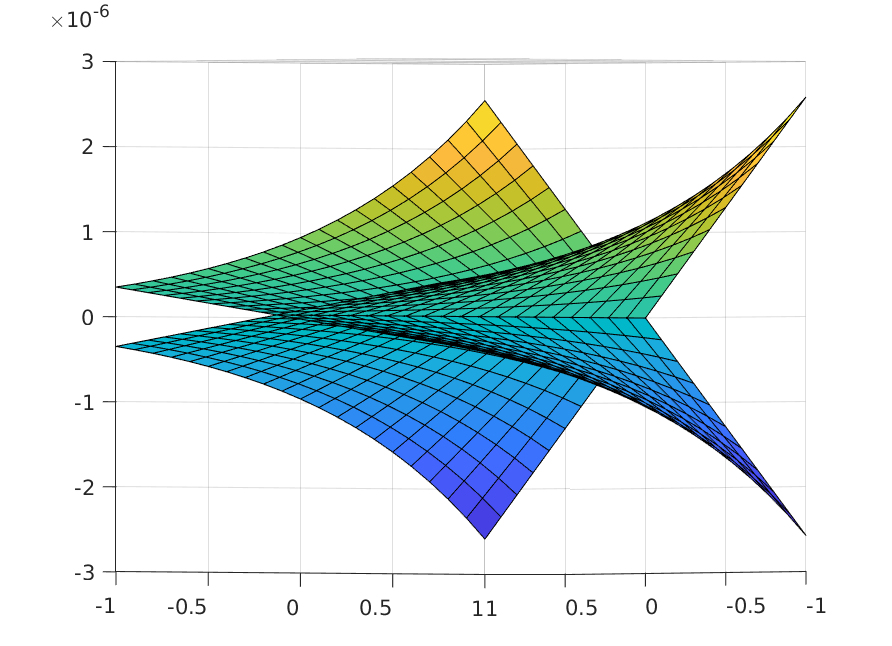}
\includegraphics[width=0.49\textwidth]{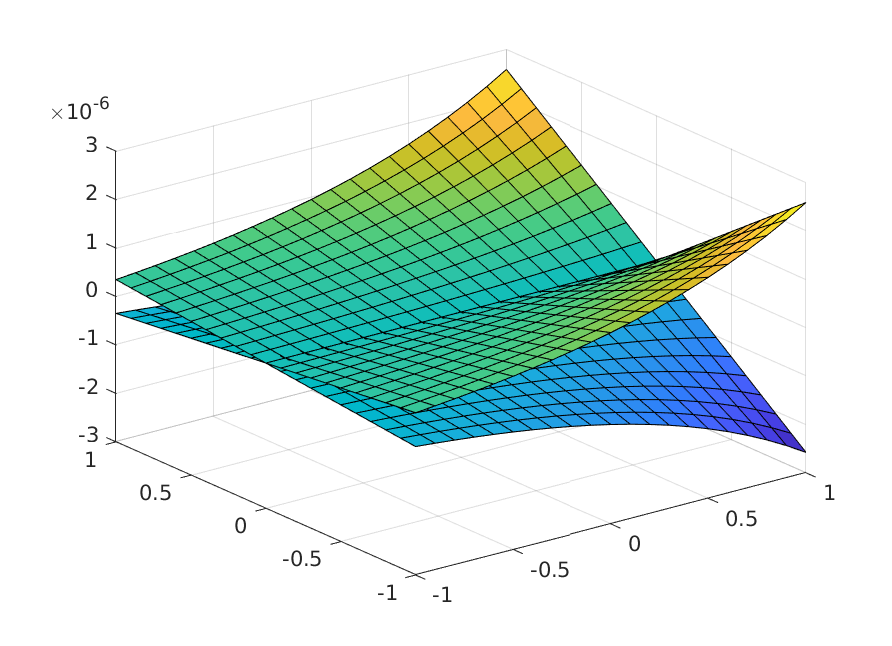}
\caption{Error graphs for Example \ref{test1}}
\label{fig:test1}
\end{figure}
\end{example} 

\begin{example}\label{test2}
In this example, we solve the integral equation 
\begin{equation*}
f(y_1,y_2)- \frac{3}{10} \int_{-1}^1 \int_{-1}^1 \sin{(x_2+x_1)}(1+x_1+y_2)
f(x_1,x_2) w(x_1,x_2) dx_1 dx_2 = g(y_1,y_2),
\end{equation*}
with $g(y_1,y_2) = \log{(2+y_2)} \sin{(\sqrt{1-y_1})}$ and
$w(x_1,x_2)=\sqrt{1-x_1^2}$ ($\alpha_1=\frac{1}{2}$,
$\beta_1=\frac{1}{2}$, $\alpha_2=0$, and $\beta_2=0$).
According to Theorem~\ref{teo:convergenza}, we fix $\gamma_1=1$,
$\delta_1=\frac{5}{4}$, $\gamma_2=\frac{2}{3}$, and $\delta_2=\frac{2}{3}$
for the weight $u$ of the function space.
Here, the exact solution $f^*$  is not available, so we approximate it by
the Nystr\"om interpolant based on the Gaussian formula with $\nn=(700,32)$.
The kernel is a smooth non-separable function whereas, for each
fixed $y_2$, $g_{y_2}(y_1) \in W_3$; therefore, by virtue of
\eqref{error}, the expected order of convergence is $O(n_1^{-3})$. 
Note that since the right-hand side has a different degree of smoothness
with respect to the two variables, we can use a number of nodes $n_2$ much
smaller than $n_1$, thus reducing the number of equations of the system.
However, the low smoothness of the right-hand side causes $n_1$ to grow.
So the size of the linear systems is moderately large, and
we solve them by the GMRES-FM method, that is, the implementation with a
factored coefficient matrix.

Table~\ref{tab:test2} reports the obtained relative errors. In this
example the good performance of the averaged interpolant in term of accuracy
is evident.
To compute it, when $\nn=(128,16)$, we have to solve
two linear systems of order $128\cdot 16=2048$, with
an error of order $10^{-11}$. The same error is produced by the
Nystr\"om method based on the Gauss rule, as reported in
Table~\ref{tab:test2}, but this requires to solve a system of order
$256\cdot 16=4096$, and so a much larger complexity and storage space.

We see that GMRES-FM converges in few iterations (reported, in parentheses, in the
second and third columns) and it is clear that the order of the system has no
effect on the speed of convergence.
In accordance with Theorem~\ref{teo:cond}, this happens because the
condition number of the coefficient matrices is small and does not depend
on the size of the systems; see the last two columns of
Table~\ref{tab:test2}.

\begin{table}[ht]
\caption{Numerical results for Example \ref{test2}.}
\label{tab:test2}
\begin{center}
\begin{tabular}{c c c c c c }
\hline
$\nn$ & $\xi^{(G)}_{\nn}(f)$ (iter) & $\xi^{(A)}_{\nn}(f)$ (iter) & $\xi^{(Avg)}_{\nn}(f)$ & $\kappa_{\infty}^{(G)}$ & $\kappa_{\infty}^{(A)}$\\
\hline
(16,16) & 	 3.28e-06 (3) & 	 2.88e-06 (3) & 	 2.04e-07 & 32.148 & 	 51.621 \\
(32,16) & 	 2.30e-07 (3) & 	 2.01e-07 (3) & 	 1.44e-08 & 36.045 & 	 54.606 \\
(64,16) & 	 1.53e-08 (3) & 	 1.34e-08 (3) & 	 9.52e-10 & 38.933 & 	 56.108 \\
(128,16) & 	 9.82e-10 (3) & 	 8.62e-10 (3) & 	 6.03e-11 &  40.998 & 	 57.277 \\
(256,16) & 	 6.13e-11 (3) & 	 5.57e-11 (3) & 	 2.78e-12 & 	 42.433 & 	 58.044 \\
(512,16) & 	 2.80e-12 (3) & 	 4.57e-12 (3) & 	 8.80e-13 & 43.442 & 	 58.591 \\
\hline
\end{tabular}
\end{center}
\end{table}
\end{example}

\begin{example}\label{test3}
Let us now consider the following equation with a separable kernel
\begin{equation*}
f(y_1,y_2)- \tfrac{3}{10} \int_{-1}^1 \int_{-1}^1
\e^{-(1+x_1)(1+y_1)-(1+y_2)(1+x_2)}  f(x_1,x_2) w(x_1,x_2) dx_1 dx_2 =
g(y_1,y_2),
\end{equation*}
with a right-hand side $g(y_1,y_2) = \cos{(3+y_2)} \, (1+y_2)^\frac{3}{2}  \, \sin\left((1-y_1)^\frac{3}{2}\right)$ characterized by a low degree of smoothness with
respect to both variables, 
and $w(x_1,x_2)=\sqrt{(1-x_1^2)(1-x_2^2)}$ with
$\alpha_1=\beta_1=\frac{1}{2}$ and $\alpha_2=\beta_2=\frac{1}{2}$.
For the weight $u$, we set $\gamma_1=\delta_1=\frac{5}{4}$ and
$\gamma_2=\delta_2=\frac{5}{4}$.

In this test, we investigate the computational time required for solving the
linear systems by Gauss's method ($PA=LU$) and the four approaches described in
the previous section: GMRES, GMRES-FM, where the coefficient matrix is
multiplied in a factored form, GMRES-SK, specially suited for the case of a
separable kernel, and the solution of Stein's equation \eqref{stein} by the
\texttt{dlyap} function of MATLAB.

As highlighted in Table \ref{tab:test3_tempi}, the application of Gauss's
method, the standard implementation of GMRES, and GMRES-FM, are unfeasible when
the system becomes moderately large.
Moreover, the first three methods go out of memory when $n_1,n_2>128$.
On the contrary, GMRES-SK has a good performance and the computational time is
comparable with that of MATLAB solver function \texttt{dlyap}. Both methods can
be applied for large problem dimensions.

\begin{table}[ht]
\caption{Computing times in seconds for Example \ref{test3}.}
\label{tab:test3_tempi}
\begin{center}
\begin{tabular}{c c c c c c}
\hline
$\nn$ & $PA=LU$ & GMRES  & GMRES-FM & GMRES-SK & \texttt{dlyap} \\
\hline 
(16,16) &    0.0784 &   0.0849  &  0.0864 &   0.0716  &  0.0705\\
(32,32) &    0.3703  &  0.3070  &  0.2928  &  0.1905  &  0.1736\\
(64,64) &    7.9212   & 7.2356  &  3.7162  &  2.8094  &  3.0727\\
(128,128) &   59.0927  & 41.9451  & 18.9175 &   8.3955  &  8.4749\\
(256,256) &      -   &   -    &   - &  26.8196 &  26.1663\\
(512,512) & -  &   -  &      - & 128.1439  & 121.2557\\
\hline
\end{tabular}
\end{center}
\end{table}
 
Table \ref{tab:test3} reports the relative errors with respect to the
approximation obtained setting $\nn=(512,512)$, which we consider exact.
The linear system is solved by the GMRES-SK method.
The averaged Nystr\"om interpolant provides 2 additional significant
digits with respect to the base interpolants starting from $\nn=(4,4)$,
until it reaches machine precision for $\nn=(128,128)$, while the
approximation based on the standard Gauss cubature rule produces the
same approximation for $\nn=(256,256)$.

It is also important to remark that, if the assertion of
Theorem~\ref{teo:braketsol} holds, the halved difference between the 
Gauss and anti-Gauss interpolants yields a bound for the
approximation error of the averaged interpolant, that is,
$$
\|f^*-\mathfrak{f_\nn}\|_\infty \leq
\frac{\|f_{\nn}-\tilde{f}_{\nn+\bm{1}}\|_\infty}{2}.
$$
Such a bound is not directly
available when a single formula is employed.

\begin{table}[ht]
\caption{Numerical results for Example \ref{test3}.}
\label{tab:test3}
\begin{center}
\begin{tabular}{c c c c }
\hline
$\nn$ & $\xi^{(G)}_{\nn}(f)$ & $\xi^{(A)}_{\nn}(f)$ & $\xi^{(Avg)}_{\nn}(f)$ \\
\hline 
(16,16) & 	  5.60e-09 & 	 5.42e-09 & 	 8.77e-11 \\
(32,32) & 	 1.05e-10 & 	 1.02e-10 & 	 1.64e-12 \\
(64,64) & 	 1.80e-12 & 	 1.74e-12 & 	 2.81e-14 \\
(128,128) & 	 2.94e-14 & 	 2.87e-14 & 	 5.29e-16 \\
(256,256) & 	 8.82e-16 & 	 9.71e-16 & 	 2.65e-16 \\
\hline
\end{tabular}
\end{center}
\end{table}
\end{example}

\begin{example}\label{test4}
In this example, we analyze the effect of a smooth right-hand side and  a
kernel which is not smooth with respect to the first variable. Hence, we apply
our method to the equation
\begin{equation*}
f(y_1,y_2)-\frac{1}{7} \int_{-1}^1 \int_{-1}^1 (x_2+y_2)|\cos
(1+x_1)|^{\frac{9}{2}} f(x_1,x_2)w(x_1,x_2) dx_1 dx_2 = g(y_1,y_2),
\end{equation*}
where $g(y_1,y_2) = \e^{y_1} \sin{y_2}$, $w(x_1,x_2)=\frac{\sqrt{1-x_2^2}}{\sqrt{1-x_1}}$ ($\alpha_1=-\frac{1}{2}$, $\beta_1=0$, $\alpha_2=\frac{1}{2}$, and $\beta_2=\frac{1}{2}$), and we fix $\gamma_1=0$, $\delta_1=\frac{1}{4}$, $\gamma_2=\frac{1}{2}$, and
$\delta_2=\frac{5}{4}$, for the weight $u$ of the function space defined
in~\eqref{u}. Also in this case the exact solution $f(x_1,x_2)$ is not
available, so we approximate it by the Nystr\"om interpolant based on the Gauss
rule with $\nn=(512,32)$.

Table \ref{tab:test4} reports in the second and third columns the
numerical errors provided by the Gauss and anti-Gauss Nystr\"om methods, respectively.
The results are better than the theoretical estimate, which is of order
$O(n_1^{-4})$.
The accuracy of the averaged interpolant improves of 1--2 significant
digits, until machine precision is reached.

\begin{table}[ht]
\caption{Numerical results for Example \ref{test4}}
\label{tab:test4}
\begin{center}
\begin{tabular}{c c c c }
\hline
$\nn$ & $\xi^{(G)}_{\nn}(f)$ & $\xi^{(A)}_{\nn}(f)$ & $\xi^{(Avg)}_{\nn}(f)$ \\
\hline 
(16,16) & 	 4.71e-09 & 	 4.92e-09 & 	 1.05e-10 \\
(32,16) & 	 8.90e-11 & 	 8.99e-11 & 	 4.97e-13 \\
(64,16) & 	 5.44e-13 & 	 6.32e-13 & 	 4.39e-14 \\
(128,16) & 	 2.49e-14 & 	 2.65e-14 & 	 8.34e-16 \\
\hline
\end{tabular}
\end{center}
\end{table}
\end{example}

\section{Conclusion and extensions}\label{sect:Concl}

This paper introduces a new anti-Gauss cubature rule and proposes its
application to the resolution of Fredholm integral equations of the
second kind defined on the square. A Nystr\"om-type method is developed,
based on Gauss and anti-Gauss cubature rules, its stability and convergence 
are analyzed, and an averaged Nystr\"om interpolant is
proposed to better approximate the solution of the problem. Numerical
tests investigate the performance of the methods and confirm the
computational advantage of the averaged Nystr\"om interpolant, in
comparison with the classical approach based on the Gauss rule.
Extensions to other averaged cubature formulae are presently being
developed by the authors.  

\section*{Acknowledgements}
The authors are members of the Gruppo Nazionale Calcolo Scientifico-Istituto Nazionale di Alta Matematica (GNCS-INdAM) and are partially supported by the the INdAM-GNCS project 2024 ``Algebra lineare numerica per problemi di grandi dimensioni: aspetti teorici e applicazioni''. L. Fermo and G. Rodriguez are also partially supported by Fondazione di
Sar\-de\-gna, Progetto biennale bando 2021, ``Computational Methods and
Networks in Civil Engineering (COMANCHE)'', by the PRIN-PNRR 2022 project no. 2022ANC8HL, and by the PRIN-PNRR 2022 project no. P20229RMLB financed by the European Union - NextGeneration EU and by the Italian Ministry of University and Research (MUR).
P. D\'iaz de Alba gratefully acknowledges Fondo Sociale Europeo REACT EU -
Programma Operativo Nazionale Ricerca e Innovazione 2014-2020 and Ministero
dell'Università e della Ricerca for the financial support. 

\bibliographystyle{plain}
\bibliography{biblio}

\end{document}